\documentclass[11pt]{amsart}

\usepackage{amssymb}
\usepackage{amsmath}
\usepackage{amsthm}
\usepackage{color} 
\usepackage{verbatim}
\usepackage{bbm}
\usepackage{dsfont}

\newtheorem{theorem}{Theorem}[section]
\newtheorem{lemma}{Lemma}[section]
\newtheorem{remark}{Remark}[section]
\newtheorem{assumption}{Assumption}[section]

\newtheorem{example}{Example}[section]
\newtheorem{definition}{Definition}[section]

\newtheorem{property}{Property}[section]

\numberwithin{equation}{section}

\newcommand{\Rd}{\mathbb{R}^d}

\newcommand{\px}{\psi_{(\xi)}}

\newcommand{\pxsop}{\frac{\psi_{(\xi)}^2}{\psi}}

\newcommand{\tr}{\operatorname{tr}}

\begin{document}

\title[Degenerate Hessian equations]
{Representation and regularity for the Dirichlet problem for real and complex degenerate Hessian equations}
\author{Wei Zhou}
\address{School of Mathematics, University of Minnesota}
\email{zhoux123@math.umn.edu}

\date{November 18, 2013}

\begin{abstract}
\noindent
We consider the Dirichlet problem for positively homogeneous, degenerate elliptic, concave (or convex) Hessian equations. Under natural and necessary conditions on the geometry of the domain, with the $C^{1,1}$ boundary data, we establish the interior $C^{1,1}$-regularity of the unique (admissible) solution, which is optimal even if the boundary data is smooth. Both real and complex cases are studied by the unified (Bellman equation) approach.
\end{abstract}

\maketitle

\tableofcontents

\section{Introduction}

\subsection{Settings and background}\label{section1} In this paper we are concerned with the representation and regularity theory for the Dirichlet problem for degenerate elliptic, concave (or convex) Hessian equations, inspired by \cite{MR806416, MR992979, MR1211724, MR1284912, MR2487853, MR3055586}.

We consider the Dirichlet problem 
\begin{equation}\label{dp}
\left\{\begin{array}{rcll}
F(u_{xx})&=&f  &\text{in }D\\ 
u_{xx}&\in&\overline\Gamma &\text{in }D\\
u&=&\varphi  &\text{on }\partial D,
\end{array}
\right.
\end{equation}
where $D$ is a bounded domain in $\Rd$ ($d\ge2$) with smooth boundary, $\Gamma$ is a connected open subset in the space $\mathbb S^d$ of real symmetric $d\times d$ matrices, $F$ is a function from $\mathbb S^d$ to $\mathbb R$, $u_{xx}$ is the Hessian matrix of $u$, and $f: D\rightarrow \mathbb R$ and $\varphi: \partial D\rightarrow \mathbb R$ are bounded and Borel measurable functions whose precise regularity assumptions will be given in each situation under consideration. 

The appearance of the second relation in the system (\ref{dp}) is due to the possible nonlinearity of the function $F$, which results in the possibly multiple components of each level set of $F$ in $\mathbb S^d$. We suppose the ellipticity of the closed set $\overline \Gamma$ in the sense of \cite{MR1284912}, i.e.
\begin{equation}
\overline\Gamma+\overline{\mathbb S_+^d}\subset\overline\Gamma, \mbox{ with }\overline{\mathbb S_+^d}=\{S\in \mathbb S^d: S\ge0\},
\end{equation}
and the $\overline{\mathbb S_+^d}$-monotonicity of $F$ on $\Gamma$, i.e.
\begin{equation}\label{dege}
F(M+S)\ge F(M),\quad \forall M\in\Gamma, S\in\overline{\mathbb S_+^d},
\end{equation}
as well as the convexity of $\Gamma$ and concavity of $F$ on $\overline\Gamma$, so that the fully nonlinear second-order partial differential equation in (\ref{dp}) is degenerate elliptic and concave.

We also take account of the complex analogy to (\ref{dp}), i.e. the Dirichlet problem for the fully nonlinear, degenerate elliptic, concave equations
\begin{equation}\label{cdp}
\left\{\begin{array}{rcll}
G(u_{z\bar z})&=&f  &\text{in }D\\ 
u_{z\bar z}&\in&\overline\Theta &\text{in }D\\
u&=&\varphi  &\text{on }\partial D,
\end{array}
\right.
\end{equation}
where $D$ is a bounded smooth domain in $\mathbb C^d$ ($d\ge2$), $\Theta$ is a elliptic, convex, connected open subset in the space $\mathbb H^d$ of Hermitian $d\times d$ matrices, $G$ is a real-valued, $\overline{\mathbb H_+^d}$-monotonic and concave function on $\overline\Theta$, $u_{z\bar z}$ is the complex Hessian  of $u$, and $f: D\rightarrow \mathbb R$ and $\varphi: \partial D\rightarrow \mathbb R$ are bounded and Borel measurable real-valued functions.

If there exists a constant $\delta>0$, such that
$$\delta\operatorname{tr}S\le F(M+S)-F(S)\le \delta^{-1}\operatorname{tr}S, \quad\forall M\in\Gamma, S\in\overline{\mathbb S_+^d},$$
then $F$ is uniformly $\delta$-nondegenerate elliptic. In the regularity theory for general fully nonlinear, uniformly nondegenerate elliptic, concave (or convex) equations, the first major breakthrough was made by Evans \cite{MR649348} and Krylov \cite{MR661144}, in which they established the interior $C^{2,\alpha}$ a priori estimate of the solution. Since then, significant contributions have been made by many mathematicians in their works, e.g. \cite{MR678347, MR688919, MR701522, MR765302, MR780073, MR984219, MR1005611}, etc.

In this paper $F$ and $G$ are always supposed to be degenerate elliptic. In this degenerate situation, generally, the $C^{2,\alpha}$-regularity of the solution breaks down, and the regularity theory for general equations, even the linear ones, is not well-developed. Meanwhile, because of the degeneracy, for the Dirichlet problem in a bounded domain, the interior regularity is dramatically influenced by and normally no better than that of the boundary data, unlike the nondegenerate cases. For quite a long time in history, and even nowadays in some situations, usually each particular degenerate equation is treated by somewhat restrictive approach which may not work for many others. Great progress on the solvability and regularity theory for the Dirichlet problem for degenerate Bellman equations was made by Krylov in \cite{MR992979} by probabilistic method and \cite{MR1325541} through analysis of PDE terms, in which under general settings, $C^{1,1}$-regularity of the solution up to the boundary was obtained with $C^{3,1}$ boundary data.

A considerable portion of the study on fully nonlinear elliptic equations has been focusing on the so-called Hessian equations. Many geometrically interesting partial differential equations are of Hessian type, e.g. real and complex Monge-Amp\`ere equations. We recall that $F$ defined on $\mathbb S^d$ (resp. $G$ defined on $\mathbb H^d$) is said Hessian if $F(M)$ (resp. $G(M)$) depends only on the eigenvalues of the symmetric matrix $M$ (resp. Hermitian matrix $M$). In other words, $F$ (resp. $G$) is Hessian if it is invariant under the action of $\mathbb O^d$ (resp. $\mathbb U^d$). In this paper we always assume $F$ and $G$ are Hessian when we work on regularity issues.
 
For nondegenerate Hessian equations, the regularity has been studied in depth under a wide variety of situations by many people. We only refer to \cite{MR806416} by Caffarelli, Nirenberg and Spruck and \cite{MR1368245} by Trudinger for real cases under general settings and  \cite{MR2128299} by S.-Y. Li for the complex general cases. 

We here pay special attention to the regularity for the Dirichlet problem for (both real and complex) degenerate Hessian equations. In general, the interior regularity of the solution is no better than that of the boundary data, and the best regularity of the (generalized) solution is $C^{1,1}$, even if the boundary data is smooth. 

For the real degenerate Monge-Amp\`ere equation (i.e. $F=\det$ in (\ref{dp})), when $D$ is strictly convex, $\varphi$ is $C^{3,1}$ and $\overline\Gamma=\overline{\mathbb S_+^d}$, the global optimal $C^{1,1}$-regularity of the convex solution has been established by Caffarelli, Nirenberg and Spruck in \cite{MR864651}, Krylov in \cite{MR992979} and P. Guan, Trudinger and X.-J. Wang in \cite{MR1687172}, under various assumptions on $f$. For the complex degenerate Monge-Amp\`ere equation (i.e. $G=\det$ in (\ref{cdp})), when $D$ is strictly pseudoconvex, $\varphi$ is $C^{3,1}$, $(f_+)^{1/d}$ is globally Lipschitz and quasi-convex, and $\overline\Theta=\overline{\mathbb H_+^d}$, the global optimal $C^{1,1}$-regularity of the plurisubharmonic solution was obtained by Krylov in \cite{MR992979}. After checking the assumptions on the main theorems in  \cite{MR992979} and \cite{MR1325541}, one understands that the global $C^{1,1}$-regularity was established for $1$-homogeneous Hessian equations under suitable conditions on the domain and the right-hand side $f$. This seems to be the only global $C^{1,1}$-regularity result for degenerate complex Hessian equations. In \cite{MR2038151}, Ivochkina, Trudinger and X.-J. Wang provided a new and shorter proof of the boundary second derivative estimate obtained in \cite{MR1325541}, which resulted in the global $C^{1,1}$-regularity for degenerate real Hessian equations under general settings. In \cite{MR2254601}, under weaker regularity assumptions on $f$, H. Dong obtained the global $C^{1,1}$-regularity result for degenerate real Hessian equations described by elementary symmetric functions of eigenvalues of $u_{xx}$.

In the aforementioned results of $C^{1,1}$-regularity up to the boundary, the associated boundary data are all assumed to be in the class $C^{3,1}$, which is necessary (see, e.g. Ex. 1, \cite{MR864651}). To obtain the optimal $C^{1,1}$-regularity of the solution in the domain, the necessary regularity of the boundary data is $C^{1,1}$. For the degenerate complex Monge-Amp\`ere equation, in \cite{MR0445006}, Bedford and Taylor showed that if $D$ is the unit ball, $\varphi\in C^{1,1}(\partial D)$, $f^{1/d}\in C^{1,1}(D)$ and $f\ge0$, then the plurisubharmonic solution is of $C^{1,1}_{loc}(D)\cap C(\bar D)$. For the real homogeneous Monge-Amp\`ere equation $\det(u_{xx})=0$, in \cite{MR766792}, Trudinger and Urbas proved that when the strictly convex domain $D$ is $C^{1,1}$ smooth, $\varphi\in C^{1,1}(\bar D)$, then the convex solution is of $C^{1,1}_{loc}(D)\cap C^{0,1}(\bar D)$. In \cite{MR878440}, Trudinger proved that the (admissible) solution of the Dirichlet problem for $k$-Hessian equations is of $C^{1,1}_{loc}(D)\cap C^{0,1}(\bar D)$ when the domain $D$ is a ball and $f^{1/k}\in C^{1,1}(\bar D)$.

Bedford and Taylor's approach in \cite{MR0445006} made use of the transitivity of the automorphism group of the unit ball in $\mathbb C^d$, which is not applicable for general strictly pseudoconvex domains. The aforementioned result in \cite{MR766792} heavily relied on the condition $f\equiv0$, as in this situation the Monge-Amp\`ere equation reduces to $\lambda_{\operatorname{min}}(u_{xx})=0$, which implies that the epigraph of the solution is the convex envelope of the epigraph of the boundary data $\varphi$. The result in \cite{MR878440} heavily relied on the condition that $D$ is a ball, so that the global barrier of the domain has a simple and symmetric formula. 

Due to lack of the interior derivative estimate for general degenerate elliptic equations, the above mentioned three works seem to be the only well-known results on the interior optimal $C^{1,1}$-regularity for degenerate Hessian equations. The only general result on interior first derivative estimate for fully nonlinear degenerate elliptic equations is due to Krylov in \cite{MR1211724}. In the author's recent work \cite{InteriorRegularityI}, an interior second derivative estimate for fully nonlinear degenerate elliptic equations was obtained by following Krylov's approach in \cite{MR1211724}. In both \cite{MR1211724} and \cite{InteriorRegularityI}, the crucial assumption that $F(\cdot, u_x, u, x)$, as a function on $\mathbb S^d$, is of Hessian type plays an important role in canceling badly-uncontrolled terms near the boundary. As the first application of the interior $C^2$ estimate obtained in \cite{InteriorRegularityI}, in \cite{InteriorRegularityII}, 
we generalized Bedford and Taylor's interior $C^{1,1}$-regularity result for degenerate complex Monge-Amp\`ere equations to any strictly pseudoconvex, $C^3$ bounded domain, and Trudinger and Urbas's interior $C^{1,1}$-regularity result for real homogeneous Monge-Amp\`ere equations to general degenerate real Monge-Amp\`ere equations (i.e. $f\ge0$), under the stronger $C^3$-regularity assumption on $\partial D$.

\subsection{Main results} In this paper we generalize the results in \cite{InteriorRegularityII} to a much broader class of equations. 

In our main Theorem \ref{ghessthm} for the Dirichlet problem (\ref{dp}), $\Gamma$ is assumed to be an elliptic, convex, $\mathbb O^d$-invariant open cone and $F$ is $1$-homogeneous, concave and Hessian on $\overline\Gamma$, satisfying $F|_{\partial \Gamma}=0$ and $F|_{\Gamma}>0$ (i.e. Assumptions \ref{a1} and \ref{a2}). By representing the system (\ref{dp}) as a Dirichlet problem for a normalized degenerate Bellman equation  with constant coefficients, in the spirit of Krylov \cite{MR1284912}, we then obtain a stochastic representation $v$ of the (probabilistic) solution (see (\ref{grp})) and establish the optimal regularity results under natural and necessary conditions on the geometry of the domain (see Assumption \ref{a3}), by applying the results in \cite{InteriorRegularityII}. Our main results are the following.
\begin{itemize}
\setlength{\itemsep}{1pt}
\item  If $\varphi,f\in C^{0,1}(\bar D)$, then $v\in C^{0,1}_{loc}(D)\cap C(\bar D)$.
\item If $\varphi\in C^{1,1}(\bar D)$, $f\in C^{0,1}(\bar D)$ and quasi-convex in $\bar D$, then $v\in C^{1,1}_{loc}(D)\cap C^{0,1}(\bar D)$ 
 is the unique strong solution to (\ref{dp}).
\end{itemize}
Although this fact is not needed in this paper, we note that when the probabilistic solution $v$ is continuous, it is the unique viscosity solution to (\ref{dp}). As discussed and explained in Remarks \ref{rmk53} and \ref{explain}, the $1$-homogeneity assumption on $F$ can be weaken by certain growth conditions of it in the direction of the identity matrix.

We are also able to treat the degenerate complex Hessian equations by a similar approach, since the interior second derivative estimate obtained in \cite{InteriorRegularityII} holds in perfect analogy in the case of complex Bellman equations. As the complex counterpart of Theorem \ref{ghessthm}, in Theorem \ref{cghessthm} for the Dirichlet problem (\ref{cdp}), where $\Theta$ is  an elliptic, convex, $\mathbb U^d$-invariant open cone in $\mathbb H^d$ and $G$ is $1$-homogeneous, concave and Hessian on $\overline\Theta$, satisfying $G|_{\partial \Theta}=0$ and $G|_{\Gamma}>0$ (i.e. Assumptions \ref{ca1} and \ref{ca2}), under natural and necessary conditions on the geometry of the domain (see Assumption \ref{ca3}), our main results on the probabilistic solution $v$ (defined by (\ref{cgrp})) are the following.
\begin{itemize}
\setlength{\itemsep}{1pt}
\item  If $\varphi,f\in C^{0,1}(\bar D)$, then $v\in C^{0,1}_{loc}(D)\cap C(\bar D)$.
\item If $\varphi, f\in C^{1,1}(\bar D)$, then $v\in C^{1,1}_{loc}(D)\cap C^{0,1}(\bar D)$ 
 is the unique strong solution to (\ref{cdp}).
\end{itemize}

A fairly large portion of this paper is concerned with specific examples. They include a wide variety of geometrically interesting equations which demonstrate the broad applicability of our theorems and interpret our assumptions from geometric viewpoints. Many of them are celebrated equations for which we establish new regularity results, while some of them have not been seen by the author in previous literature, but might possibly be interesting in the author's further study on fully nonlinear elliptic equations.

\subsection{Brief outline of the paper} The paper is organized as follows: We end Section 1 by introducing basic notation in the next subsection. 

In Sections \ref{section2} and \ref{section3} we pay particular attention to (real and complex, respectively) Hessian equations defined by Hyperbolic polynomials. Thanks to G{\aa}rding's beautiful theory  of hyperbolic polynomials, many conditions needed can just be derived, rather than assumed. We obtain representation and regularity results on them, which already generalize the results in the aforementioned previous literature in several aspects. 

Section \ref{secexa} discusses celebrated examples of two types of degenerate Hessian equations to which the theorems in Sections \ref{section2} and \ref{section3} can directly be applied. In each specific example, we interpret our assumptions and results from geometric viewpoints.

In Sections \ref{section5} and \ref{section6} we state and prove our main theorems on representation and regularity for (real and complex, respectively) degenerate Hessian equations in general settings. We feel that not putting these abstract and general results in the very beginning sections might make the paper a little bit more readable for nonexperts in the general theory of fully nonlinear partial differential equations.  

The last Section \ref{secexamore} is devoted to further examples of degenerate Hessian equations, which demonstrate the broad applicability of our main results in Sections \ref{section5} and \ref{section6}.

\subsection{Notation}\label{notation}

Throughout the paper, the summation convention for repeated indices is assumed.

We adopt the following notation on useful vector spaces, their elements, subsets and operations.
\begin{itemize}
\setlength{\itemsep}{1pt}
\item $\Rd$: the $d$-dimensional real space whose elements are denoted by $x=(x^1,\cdots,x^d)$;
\item $x\cdot y=x^iy^i$: the inner product for $x,y\in \Rd$; and $|x|^2:=x\cdot x$;
\item $\mathbb S^d, \mathbb S^d_+, \overline{\mathbb S_+^d}$: the sets of symmetric, nonnegatively-definite symmetric, positively-definite symmetric $d\times d$ matrices, respectively;
\item $\mathbb O^d$: the set of orthogonal $d\times d$ matrices;
\item $\mathbb C^d$: the $d$-dimensional complex space whose elements are denoted by
$z=(z^1,..., z^d)=(x^1+ix^{d+1},...,x^k+ix^{d+k},..., x^d+ix^{d+d})$;
\item $\operatorname{Re}z=(x^1,...,x^d), \operatorname{Im}z=(x^{d+1},...,x^{d+d}), \bar z=(\bar z^1,...,\bar z^d)=\operatorname{Re}z-i\operatorname{Im}z$: the real part, imaginary part and conjugate of $z$, respectively;
\item $\mathbb H^d, \mathbb H^d_+, \overline{\mathbb H_+^d}$: the sets of Hermitian, nonnegatively-definite Hermitian, positively-definite Hermian $d\times d$ matrices, respectively;
\item $\mathbb U^d$: the set of unitary $d\times d$ matrices;
\item $\det M, \operatorname{tr} M$: the determinant and trace of square matrix $M$;
\item $M^*$: the transpose of the matrix $M$ with real entries;
\item $\bar M^{*}$: the conjugate transpose of the matrix $M$ with complex entries.
\item $I_{d\times d}$: the identity $d\times d$ matrix;
\end{itemize}

For sufficiently smooth real-valued functions $u$ and $v$  defined on $\Rd$ and $\mathbb C^d$, respectively, we use the following notation on differentiation:
\begin{itemize}
\setlength{\itemsep}{3pt}
\item$\displaystyle u_x=(u_{x^1},\cdots,u_{x^d})$:  the gradient of $u$;
\item$\displaystyle u_{xx}=(u_{x^ix^j})_{d\times d}$: the Hessian matrix of second derivatives of $u$;
\item$\displaystyle u_{(\xi)}:=u_{x^i}\xi^i,\ u_{(\xi)(\eta)}:=u_{x^ix^j}\xi^i\eta^j, \ \forall \xi,\eta\in\Rd$;
\item$v(x):=v(z)$, with $x=(\operatorname{Re}z,\operatorname{Im}z)\in\mathbb R^{2d}$; in other words, we may view $v$ as a function from $\mathbb R^{2d}$ to $\mathbb R$;
\item$\displaystyle v_{z^k}=(v_{x^k}-iv_{x^{d+k}})/2,\ v_{\bar z^k}=(v_{x^k}+iv_{x^{d+k}})/2,\ \forall 1\le k\le d;$
\item$v_{z^k\bar z^j}=(v_{z^k})_{\bar z^j},\ v_{z\bar z}=(v_{z^k\bar z^j})_{d\times d}$;
\item$v_{(\xi)}=v_{z^k}\xi^k+v_{\bar z^k}\bar\xi^k,\ v_{(\xi)(\eta)}=(v_{(\xi)})_{(\eta)},\ \forall \xi,\eta\in\mathbb C^d;$ As a result, if $p:=(\operatorname{Re}\xi,\operatorname{Im}\xi)$ and $q:=(\operatorname{Re}\eta,\operatorname{Im}\eta)$, then
$$v_{(\xi)}(z)=v_{(p)}(x),\ v_{(\xi)(\eta)}(z)=v_{(p)(q)}(x).$$
\end{itemize}

We also speak about the following function spaces:
\begin{itemize}
\setlength{\itemsep}{1pt}
\item $C^k(D)=C^k_{loc}(D), \ k=1,2,...,\mbox{ or }\infty$: the set of functions having all continuous partial derivatives of order $\le k$  in  $D$; 
\item $C^{k}(\bar D)$: the set of functions in $C^k(D)$ all of whose partial derivatives of order $\le k$ have continuous extension to $\bar D$;
\item $C^{k,\gamma}(D)=C_{loc}^{k,\gamma}(D)$ (resp. $C^{k,\gamma}(\bar D)$): the H\"older subspace of $C^{k}(D)$ (resp. $C^k(\bar D)$) consisting of functions whose $k$-th order partial derivatives are locally H\"older continuous (resp. globally H\"older continuous) with exponent $\gamma$ in $D$, where $0<\gamma\le 1$. (In particular, when $k=0$, $\gamma=1$, we have $C^{0,1}(D)=C_{loc}^{0,1}(D)$ (resp. $C^{0,1}(\bar D)$), the space of locally Lipschitz functions (resp. globally Lipschitz functions) in $D$.)
\end{itemize}

We say a function $u$ on $D$ is $K$-quasi-convex if $u+K|x|^2/2$ is convex.

\section{Real Hessian equations defined by hyperbolic polynomials}\label{section2}

In this section we consider the Dirichlet problem for the real Hessian equation
\begin{equation}\label{hyphess}
\left\{\begin{array}{rcll}
H_m(u_{xx})&=&f^m  &\text{in }D\\ 
u_{xx}&\in&\overline\Gamma &\text{in }D\\
u&=&\varphi  &\text{on }\partial D,
\end{array}
\right.
\end{equation}
where $D$ is a bounded smooth domain in $\Rd$ ($d\ge2$), $H_m: \mathbb S^d\rightarrow\mathbb R$ is an $m$-th degree homogeneous hyperbolic polynomial with respect to the identity matrix, $\Gamma$ is the associated G{\aa}rding cone, and $f: D\rightarrow [0,\infty)$ and $\varphi: \partial D\rightarrow \mathbb R$ are continuous functions. The appearance of the second relation is due to the fully nonlinearity of $H_m$ when $m\ge 2$.  

\subsection{Hyperbolic polynomials and G{\aa}rding cones}

We start from a short review on hyperbolic polynomials, G{\aa}rding cones and their properties. 

\begin{definition}
A homogeneous polynomial $H_m$ of degree $m$ on the real vector space $\mathbb V$ is hyperbolic with respect to  $A$ if $H_m(A)>0$ and for all $M\in \mathbb V$, the polynomial (in $t$) $p(t)=H_m(tA+M)$ has $m$ real roots. Equivalently, we can write
$$H_m(tA+M)=H_m(A)\prod_{k=1}^m\Big(t+\lambda_k(M)\Big),$$
where $H_m(A)>0$, and $\lambda_k(M)=\lambda_k[H_m;A](M), 1\le k\le m$, are called the eigenvalues of $M$ for $H_m$ in the direction of $A$, or the $A$-eigenvalues of $M$ for $H_m$.
\end{definition}

 In this section, $\mathbb V$ is the  $d(d+1)/2$-dimensional real vector space $\mathbb S^d$, while in the next section, $\mathbb V=\mathbb H^d$, a $d^2$-dimensional vector space over $\mathbb R$. Throughout this article, $A=I$, the identity matrix $I$ in the corresponding space $\mathbb V$ unless otherwise specified.

For example, the determinant,
$$\det: \mathbb S^d\rightarrow \mathbb R;\quad M\mapsto \det(M),$$ is a $d$-th degree homogeneous hyperbolic polynomial (with respect to $I$). In this situation the eigenvalues $\lambda_k(M), 1\le k\le d$, are the usual eigenvalues of $M$ as a symmetric matrix, and the partial differential equation in (\ref{hyphess}) is the standard real Monge-Amp\`ere equation.
\begin{definition}
The set
$$\Gamma=\Gamma_{H_m}:=\{M\in \mathbb V: \lambda_k(M)>0, \forall 1\le k\le m\}$$
is called the G{\aa}rding cone associated to $H_m$. 
\end{definition}

For the example discussed above, the G{\aa}rding cone associated to the hyperbolic polynomial $\det$ over $\mathbb S^d$ is $\mathbb S_+^d$.
In this situation, the second relation in (\ref{hyphess}) is 
\begin{equation}\label{sense}
u_{xx}\in \overline{\mathbb S_+^d} \mbox{ in } D,
\end{equation}
 which means, we seek a convex solution to the Dirichlet problem for the real Monge-Amp\`ere equation. (See Remark \ref{gammasubh} for the generalized interpretation on the relation like (\ref{sense}) when $u\notin C^2$.)

The following properties of the hyperbolic polynomial $H_m$  and its associated G{\aa}rding cone $\Gamma_{H_m}$ are well-known: 
\begin{property}\label{ppt1}
The set $\Gamma$ is an open convex cone, with vertex at the origin. 
\end{property}
\begin{property}\label{ppt2}
 The function $H_m^{1/m}$ is concave on $\Gamma$. $H_m|_{\partial\Gamma}=0$, $H_m|_{\Gamma}>0$.
\end{property}
\begin{property}\label{ppt3}
If for each $M\in \mathbb S^d$, we list its eigenvalues in order, from small to large, i.e.
$$\lambda_1(M)\le \lambda_2(M)\le\cdots\le \lambda_m(M),$$
then for all $M\in \mathbb S^d$ and all $S\in \Gamma$, we have
$$\lambda_k(M)<\lambda_k(M+S), \quad\forall 1\le k\le m.$$
Consequently, for all $M\in \mathbb S^d$ and all $S\in \Gamma$,
\begin{equation}\label{strictmono}
H_m(M)<H_m(M+S).
\end{equation}
\end{property}
\begin{property}\label{ppt4}
Suppose that $H_m$ is hyperbolic with respect to $A$, $m>1$ and $B\in \Gamma_{H_m}$. Let $(H_m)_{(B)}$ be the directional derivative of $H_m$ in the direction $B$, i.e.,
\begin{equation}\label{dirder}
(H_m)_{(B)}(M)=\frac{d}{dt}H_m(M+tB)\Big|_{t=0}.
\end{equation}
Then $G_{m-1}:=(H_m)_{(B)}$ is an $(m-1)$-th degree homogeneous hyperbolic polynomial with respect to $A$, and
\begin{equation}\label{elmtyppt}
\Gamma_{H_m}\subset\Gamma_{G_{m-1}}.
\end{equation}
\end{property}

We refer to \cite{MR3055586} for a simple, self-contained proof of all these properties when $\mathbb V$ is any finite-dimensional real vector space.

\subsection{Assumptions and theorems}
First, due to (\ref{strictmono}), the Dirichlet problem (\ref{strictmono}) is elliptic, provided that $\Gamma$ is an elliptic set (see \cite{MR1284912}, Definition 2.1). To this end, we impose the following positivity assumption on $\Gamma$:

\begin{assumption}[Ellipticity]\label{ellip} Assume that 
\begin{equation}\label{ellipticcondition}
\overline\Gamma+\xi\xi^*\subset\overline\Gamma,\ \forall \xi\in\Rd,
\end{equation}
equivalently (since $0$ is the vertex of the open convex cone $\Gamma$),
\begin{equation}
\mathbb S^d_+\subset\Gamma.
\end{equation}
\end{assumption}

Under Assumption \ref{ellip}, we have
$$H_m(M)\le H_m(M+S), \quad\forall M\in \mathbb S^d, S\in \overline{\mathbb  S_+^d},$$
which is the definition of (degenerate) ellipticity in the fully nonlinear theory. Recall the allowance of $f$ being zero, so the Dirichlet problem (\ref{hyphess}) is degenerate elliptic if there exist $M\in\overline\Gamma$ and $S\in\overline{\mathbb S_+^d}\setminus\{0\}$, such that
$$H_m(M)=H_m(M+S).$$
Because of (\ref{strictmono}), the Dirichlet problem (\ref{hyphess}) is degenerate elliptic if and only if
\begin{equation}\label{degenerate}
\partial \Gamma\cap\overline{\mathbb S_+^d}\setminus\{0\}\ne\emptyset.
\end{equation}

Second, we suppose the partial differential equation in (\ref{hyphess}) is a Hessian equation, i.e.,
\begin{assumption}[Symmetry]
For any $O\in\mathbb O^d$, $M\in \mathbb S^d$,
\begin{equation}
H_m(M)=H_m(OMO^*).
\end{equation}
\end{assumption}

\begin{remark}
This assumption of symmetry will play an important role in our approach on studying the $C^{0,1}$- and $C^{1,1}$-regularity of the solution. If one is only interested in the continuous solution, this assumption can be dropped.
\end{remark}

We proceed by introducing the notion of strict $\Gamma$-convexity of the domain. Let $\operatorname{II}$ denote the second fundamental form of $\partial D$ with respect to the inward-pointing unit normal $\vec n$, and $\operatorname P_{\vec n}$ the orthogonal projection onto the line in the direction of $\vec n$. Our assumption on the  geometry of the domain is as follows: 
\begin{assumption}[Strict $\Gamma$-convexity of the domain]\label{hypbdry} For each $x\in \partial D$, there exists a sufficiently large scalar $t_0$, such that for all $t\ge t_0$, \begin{equation}\label{curvature}
\operatorname{II}_x+t(\operatorname P_{\vec n})_x\in \Gamma.
\end{equation}
\end{assumption}

We say that the domain $D$ is stricly $\Gamma$-convex if $D$ satisfies Assumption \ref{hypbdry}. In particular, the usual strict convexity is the strict $\mathbb S_+^d$-convexity.

Under Assumption \ref{hypbdry}, by Theorem 5.12 in \cite{MR2487853}, there exists a function $\rho\in C^{\infty}$, such that
\begin{itemize}
\setlength{\itemsep}{1pt}
\item $D=\{\rho<0\}$;
\item $\rho_x\ne0$ on $\partial D$;
\item $\exists \epsilon>0$ and $R>0$ such that for all $C\ge R$,
\begin{equation}\label{zai}
\big[C\big(\rho-\epsilon|x|^2/2\big)\big]_{xx}\in\overline\Gamma.
\end{equation}
\end{itemize} 

\begin{definition}\label{globalbarrier}
The function $\rho$ is called a global defining function for $\partial D$. We define $\psi=-C\rho$ and call it a global barrier for $D$.
\end{definition}

In order to introduce the probabilistic solution to (\ref{hyphess}), we let $w_t$ be a Wiener process of dimension $d$, and $\mathfrak A$ be the set of progressively-measurable processes $\alpha=(\alpha_t)_{t\ge0}$ with values in $\Gamma$ for all $t\ge0$. Introduce the family of controlled diffusion processes
\begin{equation}\label{contdiff}
x_t^{\alpha,x}=x+\int_0^t\sqrt{2a(\alpha_s)}dw_s,  \ \forall \alpha\in\mathfrak A,
\end{equation}
where $a=a(\alpha)$ is a $\overline{\mathbb S^d_+}$-valued function on $\Gamma$, defined by
\begin{equation}\label{aalpha}
a^{ij}(\alpha)=\frac{(H_m)_{\gamma^{ij}}(\alpha)}{\operatorname{tr}\big[(H_m)_{\gamma^{ij}}(\alpha)\big]}, \ 1\le i,j\le d,
\end{equation}
with
$$(H_m)_{\gamma^{ij}}(\alpha)=\frac{\partial}{\partial \gamma^{ij}}H_m(\gamma)\Big|_{\gamma=\alpha}.$$
The semi-positivity of $a(\alpha)$ will be proved in Lemma \ref{T}. Denote the first exit time of $x_t^{\alpha,x}$ from the domain $D$ by $\tau^{\alpha,x}$, i.e.
$$\tau^{\alpha,x}:=\inf\{t\ge0:x_t^{\alpha,x}\notin D\}.$$

Then we define a real value function $h_m=h_m(\alpha)$ on $\Gamma$ by letting
\begin{equation}\label{balpha}
h_m(\alpha):=1\Big/(H_m^{1/m})_{(I)}(\alpha)=mH^{-1/m}_m(I)\Bigg[\frac{s_m^{1/m}(\alpha)}{s^{1/(m-1)}_{m-1}(\alpha)}\Bigg]^{m-1},
\end{equation}
where $s_m(\alpha)$ and $s_{m-1}(\alpha)$ are the $m$-th and $(m-1)$-th elementary symmetric polynomials of the eigenvalues of $H_m(\alpha)$, respectively. We also define on $\Gamma\times\bar D$
\begin{equation}
f^\alpha(x)=h_m(\alpha)f(x).
\end{equation}

Now we are in a position to state our main results.

\begin{theorem}\label{drhp}
For each $x\in\bar D$, let
\begin{equation}\label{prma}
v(x)=\inf_{\alpha\in\mathfrak A}E\bigg[\varphi\big(x^{\alpha,x}_{\tau^{\alpha,x}}\big)-\int_0^{\tau^{\alpha,x}}f^{\alpha_t}\big(x_t^{\alpha,x}\big)dt\bigg].
\end{equation}
Then $v$ is Borel measurable and bounded, and for all $x\in D$,
\begin{equation}\label{0ma}
|v|\le|\varphi|_{0,\partial D}+mH_m^{-1/m}(I)|f|_{0,D}\psi.
\end{equation}

If $f,\varphi\in C^{0,1}(\bar D)$, then $v\in C^{0,1}_{loc}(D)\cap C(\bar D)$, and for a.e. $x\in D$,
\begin{equation}\label{firstma}
|v_{(\xi)}|\le N\bigg(|\xi|+\frac{|\px|}{\psi^{1/2}}\bigg),\ \forall\xi\in \Rd,
\end{equation}
where the constant $N=N(|f|_{0,1,D},|\varphi|_{0,1,D},|\psi|_{3,D},d)$.

If $f\in C^{0,1}(\bar D)$, $\varphi\in C^{1,1}(\bar{D})$ and $f$ is $K$-quasi-convex on $\bar D$, then $v\in C^{1,1}_{loc}(D)\cap C^{0,1}(\bar D)$, and for a.e. $x\in D$, 
\begin{equation}\label{secondma}
-N\frac{|\xi|^2}{\psi}\le v_{(\xi)(\xi)}\le N\bigg(|\xi|^2+\pxsop\bigg),\ \forall\xi\in \Rd,
\end{equation}
where the constant $N=N(|f|_{0,1,D},|\varphi|_{1,1,D}, |\psi|_{3,D},K, d)$. Meanwhile, $v$ is the unique solution in $C^{1,1}_{loc}(D)\cap C^{0,1}(\bar D)$ of the Dirichlet problem for the real Hessian equation:
\begin{equation}\label{realma}
\left\{\begin{array}{rcll}
H_m(u_{xx})&=&f^m  &\text{a.e. in }D\\ 
u_{xx}&\in&\overline \Gamma&\text{a.e. in }D\\
u&=&\varphi  &\text{on }\partial D.
\end{array}
\right.
\end{equation}
\end{theorem}

\begin{remark}\label{viscosity}
The readers interested in the notion of viscosity solutions understand that the probabilistic solution $u$ defined by (\ref{prma}) is the unique viscosity solution to the Dirichlet problem for the real Hessian equation (\ref{hyphess}) when it is continuous. Besides providing a stochastic representation of the unique viscosity solution, the other point of Theorem \ref{drhp} is establishing that the viscosity solution is locally Lipshitz when $f$ and $\varphi$ are Lipschitz, and it is locally $C^{1,1}$ and globally Lipschitz when $f$ and $\varphi$ are $C^{1,1}$. It is also worth mentioning that, in general, in order that $u$ is globally $C^{0,1}$ (resp. $C^{1,1}$), it is necessary that $\varphi\in C^{1,1}$ (resp. $C^{3,1}$). 
\end{remark}

\begin{remark}\label{gammasubh}
Although it is not our main focus in this paper, we note that when $u\in C_{loc}^{1,1}(D)$, by Corollary 7.5, \cite{MR2487853}, $u_{xx}\in \overline\Gamma\mbox{ a.e. in }D$ if and only if $u$ is of type of $\overline \Gamma$ (i.e. $\overline\Gamma$-subharmonic) in $D$ in the sense of Definition 4.4, \cite{MR2487853}. For example, an upper semicontinuous function $u$ is of type $\overline {\mathbb S_+^d}$ or $\overline {\mathbb S_+^d}$-subharmonic in $D$ if and only if $u$ is the usual convex function or $u\equiv-\infty$ in $D$. Therefore, for the Dirichlet problem for the real Monge-Amp\`ere equation (i.e. $H_d=\det$, $\overline\Gamma=\overline{ \mathbb S_+^d}$), when $u$ is the $C^{1,1}_{loc}(D)\cap C^{0,1}(\bar D)$-solution, the satisfaction of the relation $u_{xx}\in \overline {\mathbb S_+^d}\mbox{ a.e. in }D$ is equivalent to the convexity of $u$ in $D$.
\end{remark}

\subsection{Optimality on regularity}\label{optimality}

The following two examples show that our regularity result in Theorem \ref{drhp} is also optimal  in two senses. The first example shows that the $C^{1,1}$-regularity of the solution cannot be improved in general, even if the boundary data is smooth. The second example indicates that the $C^{1,1}$-regularity on the boundary data is necessary to obtain the interior $C^{1,1}$-regularity of the solution. 

\begin{example} [Optimality of the $C^{1,1}$-regularity, \cite{MR1687172}] Consider the unit ball $B_1(0)$ in $\mathbb R^2$ as the domain $D$ and
$$u(x_1,x_2)=\big[\max\{(x_1^2-1/2)^+,(x_2^2-1/2)^+\}\big]^2.$$
It is not hard to see that $u$ is convex and satisfies the homogeneous Monge-Amp\`ere equation with vanishing boundary data. We have $u\in C^\infty(\partial D)$, $u\in C^{1,1}(D)$, but $u\notin C^2(D)$.
\end{example}

\begin{example} [Necessity of the $C^{1,1}$-regularity of the boundary data] Still consider the unit ball $B_1(0)$ in $\mathbb R^2$ as the domain $D$. Let
$$u(x_1,x_2)=|x_1|^{2-\epsilon},$$
where $\epsilon$ is a small positive. It is not hard to see that $u$ is convex and satisfies the homogeneous Monge-Amp\`ere equation a.e. in $D$, $u\in C^{1,1-\epsilon}(\partial D)$ and $u\notin C^{1,1}(D)$.

\end{example}

\subsection{Proof of Theorem {\ref{drhp}}}

We prove Theorem \ref{drhp} by first rewriting the Hessian equation and the second relation in (\ref{hyphess}) as a  Bellman equation and then apply the regularity results established in \cite{InteriorRegularityI}. This approach was discussed generally in \cite{MR1284912}, Section 4.

The following two lemmas has been mostly established in \cite{MR2254601} when the hyperbolic polynomial $H_m$ is the $m$-th elementary symmetric polynomial of the usual eigenvalues of the matrix. See Lemmas 4.3 and 4.5 in \cite{MR2254601}. We extend them to any homogeneous hyperbolic polynomial.

\begin{lemma}\label{T}
Define $T(\alpha)=\big[T^{ij}(\alpha)\big]_{d\times d}$, where $T^{ij}(\alpha)=(H_m)_{\gamma^{ij}}(\alpha)$. Then for each $\alpha\in\Gamma$, we have 
\begin{enumerate}
\item$T(\alpha)\in  \overline{\mathbb S^d_+}$;
\item $T(O\alpha O^*)=OT(\alpha)O^*$, $\forall O\in \mathbb O^d$.
\item $\operatorname{tr}T(\alpha)=H_m(I)s_{m-1}(\alpha)>0$.
\end{enumerate}
\end{lemma}
\begin{proof}
For each  non-zero vector $\zeta\in\Rd$,
$$\zeta^*T(\alpha)\zeta=\operatorname{tr}\big[T(\alpha)\zeta\zeta^*\big]=\frac{d}{dt}H_m(\alpha+t\zeta\zeta^*)\Big|_{t=0}
=(H_m)_{(\zeta\zeta^*)}(\alpha).$$
Since $\zeta\zeta^*\in \overline{\mathbb S^d_+}\subset\overline\Gamma$, by Property \ref{ppt4}, for $\beta=\zeta
\zeta^*+\epsilon I\in\Gamma$, where $\epsilon>0$, we have
$$(H_m)_{(\beta)}(\alpha)>0
.$$
By letting $\epsilon\rightarrow0$, we have
$$\zeta^*T(\alpha)\zeta\ge 0,$$
which proves (1).

To prove (2), first note that both $H_m$ and $\Gamma$ are stable under the action of the orthogonal group. Then for any $M\in\mathbb S^d$ and $O\in\mathbb O^d$, 
\begin{align}\label{repeat}
\operatorname{tr}\big[T(O\alpha O^*)M\big]=&\frac{d}{dt}H_m(O\alpha O^*+tM)\Big|_{t=0}\\
=&\frac{d}{dt}H_m(\alpha +tO^*MO)\Big|_{t=0}\nonumber\\
=&\operatorname{tr}\big[T(\alpha)O^*MO\big]\nonumber\\
=&\operatorname{tr}\big[OT(\alpha)O^*M\big].\nonumber
\end{align}

To show (3) it suffices to notice that both sides of the equation are equal to $(H_m)_{(I)}(\alpha)$. Since $\alpha\in \Gamma$, $\alpha$ is also in the G{\aa}rding cone of $(H_m)_{(I)}$. As a result, $(H_m)_{(I)}(\alpha)>0.$
\end{proof}

\begin{lemma}\label{equiveqn}
For any $c\ge0$, the system
\begin{equation}\label{equivhyp}
\left\{\begin{array}{rcl}
H_m(\gamma)&=&c^m\\
\gamma&\in&\overline\Gamma
\end{array}
\right.
\end{equation}
is equivalent to the  Bellman equation
\begin{equation}\label{equivbell}
\inf_{\alpha\in\Gamma}\big\{a^{ij}(\alpha)\gamma_{ij}-h_m(\alpha)c\big\}=0,
\end{equation}
where $a(\alpha)$ and $h_m(\alpha)$ are given by (\ref{aalpha}) and (\ref{balpha}), respectively.
\end{lemma}

\begin{proof}[Proof of (\ref{equivhyp})$\Rightarrow$(\ref{equivbell})]

Since $H_m^{1/m}$ is concave in $\Gamma$, for each $\gamma\in\Gamma$, we have
\begin{equation*}
H_m^{1/m}(\gamma)=\inf_{\alpha\in\Gamma}\Big\{(H_m^{1/m})_{\gamma^{ij}}(\alpha)(\gamma_{ij}-\alpha_{ij})+H_m^{1/m}(\alpha)\Big\}.
\end{equation*}
By Euler's homogeneous function theorem, we have
$$(H_m^{1/m})_{\gamma^{ij}}(\alpha)\alpha_{ij}=H_m^{1/m}(\alpha).$$
Therefore
\begin{equation}\label{bellhomo}
H_m^{1/m}(\gamma)=\inf_{\alpha\in\Gamma}\big\{(H_m^{1/m})_{\gamma^{ij}}(\alpha)\gamma_{ij}\big\}.
\end{equation}
It follows that on $\Gamma$, the equation $H_m(\gamma)=c^m$ can be rewritten as
$$\inf_{\alpha\in\Gamma}\big\{m^{-1}H_m^{1/m-1}(\alpha)(H_m)_{\gamma^{ij}}(\alpha)\gamma_{ij}-c\big\}=0.$$

Since for each $\gamma\in\Gamma$, the infimum of the expression in the bracket is attained at $\alpha=\gamma$, we obtain
$$(H_m)_{\gamma^{ij}}(\gamma)\gamma_{ij}=cmH_m^{1-1/m}(\gamma).$$
By Lemma \ref{T} (3), for each $\gamma\in\Gamma$, we can divide both sides by $\operatorname{tr}\big[(H_m)_{\gamma^{ij}}(\gamma)\big]$,  so
\begin{align*}
a^{ij}(\gamma)\gamma_{ij}=&cm\frac{H_m^{1-1/m}(\gamma)}{\operatorname{tr}\big[(H_m)_{\gamma^{ij}}(\gamma)\big]}\\
=&cm\frac{\big[H_m(I)s_m(\gamma)\big]^{1-1/m}}{H_m(I)s_{m-1}(\gamma)}\\
=&ch_m(\gamma).
\end{align*}
Therefore for each $\gamma\in\Gamma$, Equation (\ref{equivbell}) is true since again the infimum there is attained at $\alpha=\gamma$.

If $\gamma\in\partial\Gamma$, then $c=0$. For each $\epsilon>0$, we have $\gamma+\epsilon I\in \Gamma$. Therefore,
\begin{align*}
a^{ij}(\gamma+\epsilon I)\gamma_{ij}=&\frac{(H_m)_{\gamma^{ij}}(\gamma+\epsilon I)\gamma_{ij}}{\operatorname{tr}\big[(H_m)_{\gamma^{ij}}(\gamma+\epsilon I)\big]}\\
=&\frac{(H_m)_{\gamma^{ij}}(\gamma+\epsilon I)(\gamma_{ij}+\epsilon\delta_{ij}-\epsilon\delta_{ij})}{\operatorname{tr}\big[(H_m)_{\gamma^{ij}}(\gamma+\epsilon I)\big]}\\
=&m\frac{H_m(\gamma+\epsilon I)}{\operatorname{tr}\big[(H_m)_{\gamma^{ij}}(\gamma+\epsilon I)\big]}-\epsilon.
\end{align*}
Since $p(\epsilon)=H_m(\gamma+\epsilon I)$ and $q(\epsilon)=\operatorname{tr}\big[(H_m)_{\gamma^{ij}}(\gamma+\epsilon I)\big]$ are $m$-th degree and $(m-1)$-th degree polynomials, we have
$$\lim_{\epsilon\rightarrow0}a^{ij}(\gamma+\epsilon I)\gamma_{ij}=0,$$
which implies that Equation (\ref{equivbell}) is true.
\end{proof}

\begin{proof}[Proof of (\ref{equivbell})$\Rightarrow$(\ref{equivhyp})]
 
We first establish the relation $\gamma\in\overline\Gamma$. By (\ref{bellhomo}),
\begin{align*}
\overline\Gamma=&\{\gamma\in\mathbb S^d: (H_m^{1/m})_{\gamma^{ij}}(\alpha)\gamma_{ij}\ge0,\forall \alpha\in\Gamma\}\\
=&\{\gamma\in\mathbb S^d: a^{ij}(\alpha)\gamma_{ij}\ge0,\forall \alpha\in\Gamma\}.
\end{align*}
On the other hand, since $c\ge0$, $h_m(\alpha)>0$ on $\Gamma$, given any $\gamma$ satisfying (\ref{equivbell}), we have
$$a^{ij}(\alpha)\gamma_{ij}\ge0, \quad\forall\alpha\in\Gamma.$$
Thus the relation is established.

Then we verify the equation $H_m(\gamma)=c$. Given any $\gamma\in\overline\Gamma$, we notice that 
$$\lim_{t\rightarrow+\infty}H_m(tI+\gamma)=\lim_{t\rightarrow+\infty}t^mH_m(I+t^{-1}\gamma)=+\infty,$$
and that there exists $t_0\le0$, such that $t_0I+\gamma\in\partial\Gamma$ and $H_m(t_0I+\gamma)=0$. By the continuity of $p(t)=H_m(tI+\gamma)$, for each $c\ge0$, there exists a scalar $t_1\in [t_0,\infty)$, such that $t_1I+\gamma\in\overline\Gamma$ and $H_m(t_1I+\gamma)=c^m$.

From the proof of (\ref{equivhyp})$\Rightarrow$(\ref{equivbell}), we know that
\begin{align*}
\inf_{\alpha\in\Gamma}\Big\{a^{ij}(\alpha)(t_1\delta_{ij}+\gamma_{ij})-h(\alpha)c\Big\}=&0\\
\inf_{\alpha\in\Gamma}\Big\{a^{ij}(\alpha)\gamma_{ij}-h(\alpha)c\Big\}=&-t_1.
\end{align*}
Comparison with (\ref{equivbell}) gives $t_1=0$, so Equation (\ref{equivhyp}) holds.
\end{proof}


We next study the properties of the global barrier $\psi$ given in Definition \ref{globalbarrier}.

\begin{lemma}\label{psi}
For sufficiently large constant $C$, the global barrier $\psi$ given in Definition \ref{globalbarrier} satisfies the following conditions:
\begin{enumerate}
\item $D=\{\psi>0\}$;
\item $|\psi_x|\ge 1$ on $\partial D$;
\item $\displaystyle\sup_{\alpha\in\Gamma}\big(a_{ij}(\alpha)\psi_{x^ix^j}\big)\le-1$ in $\bar D$.
\end{enumerate}
\end{lemma}
\begin{proof}
Recall that $\psi=-C\rho$. The truthness of (1) and (2) are trivial. To prove (3) we notice that by (\ref{zai}),
$$H_m\Big(C(\rho-\frac{\epsilon}{2}|x|^2)\Big)\ge 0.$$
By Lemma \ref{equiveqn}, 
$$\inf_{\alpha\in\Gamma}\Big\{a^{ij}(\alpha)\Big(C(\rho-\frac{\epsilon}{2}|x|^2)\Big)_{x_ix_j}\Big\}\ge0.$$
It follows that
$$\inf_{\alpha\in\Gamma}\big\{a^{ij}(\alpha)(C\rho)_{x_ix_j}\big\}\ge C\epsilon,$$
which implies the truthness of (3) by letting $C$ be sufficiently large.

\end{proof}

Now we are ready to prove Theorem \ref{drhp}.

\begin{proof}[Proof of Theorem \ref{drhp}]
We apply Theorems 2.1, 2,2 and 2.3 in \cite{InteriorRegularityI} with
$$A=\Gamma,\qquad a^\alpha=\Bigg(\frac{(H_m)_{\gamma^{ij}}(\alpha)}{\operatorname{tr}\big[(H_m)_{\gamma^{ij}}(\alpha)\big]}\Bigg)_{1\le i,j\le d},$$
$$\sigma^\alpha=\sqrt{2a^\alpha},\qquad b^\alpha=0,\qquad c^\alpha=0,$$
$$f^\alpha(x)=mH^{-1/m}_m(I)\Bigg[\frac{s_m^{1/m}(\alpha)}{s^{1/(m-1)}_{m-1}(\alpha)}\Bigg]^{m-1}f(x),\qquad g(x)=-\varphi(x).$$

Based on these substitutions, it is not hard to see that $v$ defined by (\ref{prma}) equals $-v$ defined by (2.3) in \cite{InteriorRegularityI}.

By Maclaurin's inequality, 
$$\frac{s_m^{1/m}(\alpha)}{s^{1/(m-1)}_{m-1}(\alpha)}\le{\binom{m}{m-1}}^{-1/(m-1)}{\binom{m}{m}}^{1/m}=m^{-1/(m-1)},$$
which implies that
$$m\Bigg[\frac{s_m^{1/m}(\alpha)}{s^{1/(m-1)}_{m-1}(\alpha)}\Bigg]^{m-1}\le1.$$

By Lemma \ref{psi}, Assumption 2.1 in \cite{InteriorRegularityII} holds. Recall that in \cite{InteriorRegularityI}, we define $\mathbb A=\{a^\alpha:\alpha\in A\}$. By Lemma \ref{T}, we see that
$$\mathbb A\subset\{M\in\overline{\mathbb S_+^d}: \operatorname{tr} M=1\},$$
and for each $O\in\mathbb O^d$, we have
$$O\mathbb A O^*=\mathbb A$$
and
$$h(O\alpha O^*)=h(\alpha),$$
so Assumption 2.2 in \cite{InteriorRegularityI} holds.

To verify to weak non-degeneracy we consider $a^{\alpha_0}=(1/d)I_{d\times d}\in\mathbb A$, then we have
$$\mu=\inf_{|\zeta|=1}\sup_{\alpha\in A}(a^\alpha)_{ij}\zeta^i\zeta^j\ge\inf_{|\zeta|=1}(a^{\alpha_0})_{ij}\zeta^i\zeta^j=1/d>0.$$

Therefore we can apply Theorems 2.1, 2.2 and 2.3 in \cite{InteriorRegularityI} to obtain Theorem \ref{drhp}.

\end{proof}

\section{Complex Hessian equations defined by hyperbolic polynomials}\label{section3}

This section is the complex counterpart of the previous section. Refer to Subsection \ref{notation} for the notation we adopted in complex analysis. 

\subsection{Assumptions and theorems}
We consider the Dirichlet problem for the complex Hessian equation
\begin{equation}\label{chyphess}
\left\{\begin{array}{rcll}
G_m(u_{z\bar z})&=&f^m  &\text{in }D\\ 
u_{z\bar z}&\in&\overline\Theta&\text{in }D\\
u&=&\varphi  &\text{on }\partial D,
\end{array}
\right.
\end{equation}
where $D$ is a bounded smooth domain in $\mathbb C^d$ ($d\ge2$), $G_m: \mathbb H^d\rightarrow\mathbb R$ is an $m$-th degree homogeneous hyperbolic polynomial with respect to the identity matrix, $\Theta$ is the associated G{\aa}rding cone, and $f: D\rightarrow [0,\infty)$ and $\varphi: \partial D\rightarrow \mathbb R$ are continuous functions.


Note that $\mathbb H_d$ is a $d^2$-dimensional vector space over $\mathbb R$, all of the elementary properties of hyperbolic polynomials and G{\aa}rding cones mentioned in Section \ref{section2} are satisfied. See, e.g., \cite{MR3055586}. In order that the Dirichlet problem is elliptic, we impose the following positivity assumption on $\Theta$:

\begin{assumption}[Ellipticity]\label{cmono} Assume that
$$\mathbb H^d_+\subset\Theta.$$
\end{assumption}

We also suppose the partial differential equation in (\ref{chyphess}) is a Hessian equation, i.e.,
\begin{assumption}[Symmetry]
For any $U\in\mathbb U^d$, $M\in \mathbb H^d$,
\begin{equation}
G_m(M)=G_m(UM\bar U^*).
\end{equation}
\end{assumption}

In order that the Dirichlet problem (\ref{chyphess}) is solvable, our assumption on the geometry of the boundary of the domain is the following:

\begin{assumption}[Strict $\Theta$-pseudoconvexity of the domain]\label{chypbdry} For each $z\in \partial D$, there exists a sufficiently large scalar $t_0$, such that for all $t\ge t_0$, we have $\operatorname{L}_z+t(\operatorname P_{\vec n})_z\in \Theta$, where $\operatorname{L}_z$ is the Levi form of the defining function of $\partial D$ around $z$ and $(\operatorname P_{\vec n})_z$ is the orthogonal projection onto the complex line in the direction of the inward-pointing normal $\vec n$ at $z$.
\end{assumption}

If we take the $d$-dimensional hyperbolic polynomial $\det: \mathbb H_d\rightarrow \mathbb R$ as an example, then the partial differential equation in (\ref{chyphess}) is the classical complex Monge-Amp\`ere equation, the second relation in (\ref{chyphess}) means that we seek plurisubharmonic solutions. In this situation, Assumption \ref{cmono} holds since $\mathbb H_+^d=\Theta$, while Assumption \ref{chypbdry} means that we consider the domain which is strictly pseudo-convex. Our settings also cover some other interesting complex Hessian equations which will be discussed in subsequent sections.

It is also worth point out that, similarly to the real cases, the elliptic equation in (\ref{chyphess}) is degenerate if and only if
\begin{equation}
\partial \Theta\cap \overline{\mathbb H_+^d}\setminus\{0\}\ne \emptyset.
\end{equation}

To introduce the probabilistic solution, we let $W_t$ be the normalize complex Wiener process of dimension $d$, i.e. a $d$-dimensional stochastic process $W_t=(W_t^1,..., W_t^d)$ with values in $\mathbb C^d$ given by 
$$W_t^j=\frac{1}{\sqrt 2}\Big(w_t^{j,1}+iw_t^{j,2}\Big), \ t\ge0, \ 1\le j\le d,$$
where the processes $(w_t^{j,1},w_t^{j,2})_{1\le j\le d}$ are independent real Wiener processes. Let $\mathfrak B$ be the set of progressively-measurable processes $\beta=(\beta_t)_{t\ge0}$ with values in $\Theta$ for all $t\ge0$. Introduce a family of controlled diffusion processes
\begin{equation}\label{ccontdiff}
z_t^{\beta,z}=z+\int_0^t\sqrt{b(\beta_s)}dW_s,  \ \forall \beta\in\mathfrak B,
\end{equation}
where $b=b(\beta)$ is a $\overline{\mathbb H^d_+}$-valued function on $\Theta$, defined by
\begin{equation}\label{caalpha}
b^{ij}(\beta)=\frac{(G_m)_{\theta^{ij}}(\beta)}{\operatorname{tr}\big[(G_m)_{\theta^{ij}}(\beta)\big]}, \ 1\le i,j\le d.
\end{equation}
Then we define a real-valued function $l_m=l_m(\beta)$ on $\Theta$ by letting
\begin{equation}\label{cbalpha}
l_m(\alpha):=1\Big/(G_{m}^{1/m})_{(I)}(\beta)=mG^{-1/m}_m(I)\Bigg[\frac{r_m^{1/m}(\beta)}{r^{1/(m-1)}_{m-1}(\beta)}\Bigg]^{m-1},
\end{equation}
where $r_m(\beta)$ and $r_{m-1}(\beta)$ are the $m$-th and $(m-1)$-th elementary symmetric polynomials of the eigenvalues of $G_m(\beta)$, respectively. We also define
\begin{equation}
f^\beta(z)=l_m(\beta)f(z).
\end{equation}
Denote the first exit time of $z_t^{\beta,z}$ from the domain $D$ by $\tau^{\beta,z}$.

Under Assumption \ref{chypbdry}, there exists a global barrier $\psi$ satisfying the following conditions:
\begin{enumerate}
\item $D=\{\psi>0\}$;
\item $|\psi_z|\ge 1$ on $\partial D$;
\item $\displaystyle\sup_{\beta\in\Theta}\big(b_{ij}(\beta)\psi_{z^i\bar z^j}\big)\le-1$ in $\bar D$.
\end{enumerate} 

We can now state and prove our results.

\begin{theorem}\label{cdrhp}
For each $z\in\bar D$, let
\begin{equation}\label{cprma}
v(z)=\inf_{\beta\in\mathfrak B}E\bigg[\varphi\big(z^{\beta,z}_{\tau^{\beta,z}}\big)-\int_0^{\tau^{\beta,z}}f^{\beta_t}\big(z_t^{\beta,z}\big)dt\bigg].
\end{equation}
Then $v$ is Borel measurable and bounded, and for all $z\in D$,
\begin{equation}\label{c0ma}
|v|\le|\varphi|_{0,\partial D}+mG_m^{-1/m}(I)|f|_{0,D}\psi.
\end{equation}

If $f,\varphi\in C^{0,1}(\bar D)$, then $v\in C^{0,1}_{loc}(D)\cap C(\bar D)$, and for a.e. $z\in D$,
\begin{equation}\label{cfirstma}
|v_{(\xi)}|\le N\bigg(|\xi|+\frac{|\px|}{\psi^{1/2}}\bigg),\ \forall\xi\in \mathbb C^d,
\end{equation}
where the constant $N=N(|f|_{0,1,D},|\varphi|_{0,1,D},|\psi|_{3,D},d)$.

If $f, \varphi\in C^{0,1}(\bar D)$, then $v\in C^{1,1}_{loc}(D)\cap C^{0,1}(\bar D)$, and for a.e. $z\in D$, 
\begin{equation}\label{csecondma}
-N\frac{|\xi|^2}{\psi}\le v_{(\xi)(\xi)}\le N\bigg(|\xi|^2+\pxsop\bigg),\ \forall\xi\in \Rd,
\end{equation}
where the constant $N=N(|f|_{0,1,D},|\varphi|_{1,1,D}, |\psi|_{3,D}, d)$. Meanwhile, $v$ is the unique solution in $C^{1,1}_{loc}(D)\cap C^{0,1}(\bar D)$ of the Dirichlet problem for the degenerate complex Hessian equation:
\begin{equation}\label{complexma}
\left\{\begin{array}{rcll}
G_m(u_{z\bar z})&=&f^m  &\text{a.e. in }D\\ 
u_{z\bar z}&\in&\overline\Theta&\text{a.e. in }D\\
u&=&\varphi  &\text{on }\partial D.
\end{array}
\right.
\end{equation}
\end{theorem}

\begin{remark}
The regularity results in the complex case are still optimal in the same sense discussed in Subsection \ref{optimality} by constructing similar examples. Also, the readers interested in Remarks \ref{viscosity} and \ref{gammasubh} understand that similar remarks hold in the complex case. 
\end{remark}

\subsection{Proof of Theorem {\ref{cdrhp}}}\label{subsection32} We divide the proof of Theorem \ref{cdrhp} into the following three steps.
\begin{enumerate}
\item View $v(z)$ given by (\ref{cprma}) as a function $v(x)$ on $\mathbb R^{2d}$ and represent it as the value function of a stochastic control problem in $\mathbb R^{2d}$.
\item Utilize the results in \cite{InteriorRegularityI} to obtain the interior regularity for $v(x)$.
\item Verify that the associated dynamic programming equation of $v(x)$, along with the boundary condition, is equivalent to (\ref{chyphess}).
\end{enumerate}
\begin{proof}[Step 1]
In order to make use of the deduction in the real cases, we first define the following homomorphisms:
$$\Phi: \mathbb C^d\rightarrow\mathbb R^{2d}; \quad z\mapsto 
\begin{pmatrix}
\operatorname{Re}z\\
\operatorname{Im}z
\end{pmatrix}
$$
$$
\Phi: \mathbb C^{d\times d}\rightarrow\mathbb R^{2d\times 2d}; \quad\beta\mapsto 
\begin{pmatrix}
\operatorname{Re}\beta&\operatorname{Im}\beta\\
-\operatorname{Im}\beta&\operatorname{Re}\beta
\end{pmatrix}.
$$
To rewrite the value function in (\ref{cprma}) as a function on $\mathbb R^{2d\times 2d}$, we notice that
$$\Phi( z_t^{\beta,z})=\Phi z+\int_0^t\frac{1}{\sqrt 2}\Big(\Phi\sqrt{b({\beta_s})}\Big)dw_s,$$
where $w_t$ is a Wiener process of dimension $2d$. Denote
$$\alpha=\Phi\beta, \qquad\mathfrak A=\Phi\mathfrak B=\{\Phi\beta:\beta\in\mathfrak B\},$$
$$\Phi\mathbb H^d=\{\Phi\beta:\beta\in\mathbb H^d\}.$$
Note that $\Phi\mathbb H^d\subset\mathbb S^{2d}$ is a $d^2$-dimensional real vector space. On $\Phi\mathbb H^d$, define $$H_{2m}(\alpha)=G_m^2(\Phi^{-1}\alpha).$$
Then $H_{2m}$ is a $2m$-degree hyperbolic polynomial over $\Phi\mathbb H^d$ with respect to $I_{2d\times2d}$ and for each $\beta\in \mathbb H^d$
\allowdisplaybreaks
\begin{align*}
r_m(\beta)=&s_{2m}^{1/2}(\alpha),\\
r_{m-1}(\beta)=&\frac{d}{dt}r_m(\beta+tI_{d\times d})\Big|_{t=0}\\
=&\frac{d}{dt}s_{2m}^{1/2}(\alpha+tI_{2d\times 2d})\Big|_{t=0}\\
=&\frac{1}{2}s_{2m}^{-1/2}(\alpha)s_{2m-1}(\alpha),\\
\end{align*}
which implies that
\begin{align*}
m\Bigg[\frac{r_m^{1/m}(\beta)}{r^{1/(m-1)}_{m-1}(\beta)}\Bigg]^{m-1}=&2m\Bigg[\frac{s_{2m}^{1/(2m)}(\alpha)}{s^{1/(2m-1)}_{2m-1}(\alpha)}\Bigg]^{2m-1}\\
l_m(\beta)=&h_{2m}(\alpha).
\end{align*}
We also note that
$$\Phi(b(\beta))=\Bigg(\frac{2(H_{2m})_{\gamma^{ij}}(\alpha)}{\operatorname{tr}\big[(H_{2m})_{\gamma^{ij}}(\alpha)\big]}\Bigg)_{1\le i,j\le 2d}:=4a(\alpha).$$
Therefore we can rewrite (\ref{cprma}) as
\begin{equation}
v(x)=\inf_{\alpha\in\mathfrak A}E\bigg[\varphi(x^{\alpha,x}_{\tau^{\alpha,x}})-\int_0^{\tau^{\alpha,x}}f^{\alpha_t}(x^{\alpha,x}_t)dt\bigg],
\end{equation}
where
$$x_t^{\alpha,x}=x+\int_0^t\sqrt{2a(\alpha_s)}dw_s,$$
$$f^{\alpha}(x)=h_{2m}(\alpha)f(x).$$
\end{proof}

\begin{proof}[Step 2]
Let
$$\mathbb B=\{b^\beta: \beta\in\Theta\},\qquad\mathbb A=\Phi \mathbb B=\{\Phi (b^\beta): b^\beta\in\mathbb  B\}.$$
It is worth mentioning that since
$$\mathbb A\subset\bigg\{
\begin{pmatrix}
S&T\\
-T&S
\end{pmatrix}
: S\in\overline{\mathbb S^d_+}, \operatorname{tr}(S)=1;T \mbox{ is skew symmetric}\bigg\}
,$$
in general, the relation
$$O\mathbb A O^*=\mathbb A$$
doesn't hold for all $O\in\mathbb O^{2d}$. However, for any $U\in\mathbb U^d$, we have
$$U\mathbb B\bar U^*=\mathbb B,$$
which can play a role of Assumption 2.2 in \cite{InteriorRegularityI}. Indeed, we observe that
$$\mathbb A=\Phi\mathbb B=\Phi(U\mathbb B\bar U^*)=\Phi U\Phi\mathbb B\Phi(\bar U^*)=(\Phi U)\mathbb A(\Phi U)^*.$$
Moreover, we note that if $Q\in \mathbb C^{d\times d}$ is skew Hermitian, then $e^Q$ is unitary, and therefore
$e^{\Phi Q}=\Phi(e^Q)$ satisfies
$$(e^{\Phi Q})\mathbb A (e^{\Phi Q})^*=\mathbb A.$$
Therefore, in order to apply Theorems 2.1, 2.2 and 2.3 in \cite{InteriorRegularityI}, it suffice to find a suitable matrix function $P=P(x,\xi)$ from $D\times\mathbb R^{2d}$ to $\mathbb R^{2d\times2d}$, such that $P$ can be expressed as $\Phi Q$, where $Q$ is skew Hermitian, and $P(x,\xi)$ satisfies all properties it has in the proof of Lemma 7.1 in \cite{InteriorRegularityI}. 

To construct $P$ let us start from looking at Equation (7.1) in \cite{InteriorRegularityI}, the most crucial property it satisfies in the proof of Lemma 7.1 in \cite{InteriorRegularityI}.
We define
$$\chi: D_\delta^\lambda\times \mathbb C^d\rightarrow \mathbb C; (z,\xi)\mapsto -\frac{\psi_{\bar z^k}(\psi_{z^k})_{(\xi)}}{|\psi_{\bar z}|^2}$$
and
$$R: D_\delta^\lambda\times \mathbb C^d\rightarrow \mathbb C^{d\times d}; (z,\xi)\mapsto \bigg(\frac{(\psi_{z^k})_{(\xi)}\psi_{\bar z^j} -(\psi_{\bar z^j})_{(\xi)}\psi_{z^k}}{|\psi_{\bar z}|^2}\bigg)_{1\le j,k\le d},$$
which are analogous to $\rho$ and $P$ in Lemma 7.1 in \cite{InteriorRegularityI}.

We claim that
\begin{equation}\label{nm}
(\psi_{\bar z})_{(\xi)}+R\psi_{\bar z}+\chi\psi_{\bar z}=0.
\end{equation}
Indeed, we have
\begin{align*}
&\Big[(\psi_{\bar z})_{(\xi)}+R\psi_{\bar z}+\chi\psi_{\bar z}\Big]^j\\
=&(\psi_{\bar z^j})_{(\xi)}+\frac{(\psi_{ z^k})_{(\xi)}\psi_{\bar z^j}-(\psi_{\bar z^j})_{(\xi)}\psi_{z^k}}{\psi_z\psi_{\bar z}}\psi_{\bar z^k}-\frac{\psi_{\bar z^k}(\psi_{ z^k})_{(\xi)}}{\psi_z\psi_{\bar z}}\psi_{\bar z^j}\\
=&(\psi_{\bar z^j})_{(\xi)}\Big[1-\frac{\psi_{z^k}\psi_{\bar z^k}}{\psi_z\psi_{\bar z}}\Big]\\
=&0.
\end{align*}
Next, we notice that $\chi$ is not real in general, so we decompose it as $\chi=\varrho+i\varkappa$, where $\varrho$ and $\varkappa$ are real valued. If we denote $R+i\varkappa I$ as $Q$, the equation (\ref{nm}) can be rewritten as
\begin{equation}\label{cimportant}
(\psi_{\bar z})_{(\xi)}+Q\psi_{\bar z}+\varrho\psi_{\bar z}=0.
\end{equation}
We emphasis that $\varrho$ is real and $Q$ is skew Hermitian.
From (\ref{cimportant}) and the fact that $\psi_x=2\Phi(\psi_{\bar z})$, we obtain
\begin{align*}
0=&\Big(\alpha\alpha^*\big((\psi_{\bar z})_{(\xi)}+Q\psi_{\bar z}+\varrho\psi_{\bar z}\big),(\psi_{\bar z})_{(\xi)}+Q\psi_{\bar z}+\varrho\psi_{\bar z}\Big)\\
=&\bigg(\Phi\Big(\alpha\alpha^*\big((\psi_{\bar z})_{(\xi)}+Q\psi_{\bar z}+\varrho\psi_{\bar z}\big)\Big),\Phi\Big((\psi_{\bar z})_{(\xi)}+Q\psi_{\bar z}+\varrho\psi_{\bar z}\Big)\bigg)\\
=&\frac{1}{4}\Big(\beta\beta^*\big((\psi_x)_{(\xi)}+(\Phi Q)\psi_x+\varrho\psi_x \big),(\psi_x)_{(\xi)}+(\Phi Q)\psi_x+\varrho\psi_x\Big)
\end{align*}
Therefore if we let $P=\Phi Q$,
$$\psi_{(\xi)(\beta^k)}+\varrho\psi_{(\beta^k)}+\psi_{(P\beta^k)}=0.$$
It is also not hard to see that $P$ and $\varrho$ we define here satisfy all the other property in Lemma 7.1 in \cite{InteriorRegularityI}. As a result, we can apply Theorem 2.1, 2.2 and 2.3 in \cite{InteriorRegularityI} to obtain all results stated in Theorem \ref{cdrhp} on $v$ defined by (\ref{cprma}).
\end{proof}

\begin{proof}[Step 3]
It remains to verify that the associated dynamic programing equation is equivalent to the complex Hessian equation. To write down the real Bellman equation we note that its diffusion term is $\sqrt{2a(\alpha)}$. Hence the associated dynamic programing equation is the real Bellman equation
\begin{equation}
\inf_{\alpha\in \Phi\Theta}\Big\{\operatorname{tr}\big[a(\alpha)v_{xx}\big]-h_{2m}(\alpha)f\Big\}=0,
\end{equation}
which is equivalent to 
\begin{equation}
\inf_{\beta\in \Theta}\Big\{\frac{1}{4}\operatorname{tr}\big[(\Phi b(\beta))v_{xx})\big]-l_{m}(\beta)f\Big\}=0.
\end{equation}
To write down the corresponding complex Bellman equation, it suffices to notice that
$$\operatorname{tr}\big[(\Phi b(\beta))v_{xx}\big]=4\operatorname{tr}\big[b(\beta)v_{z\bar z}\big].$$
Therefore the complex Bellman equation is
\begin{equation}\label{cbellmanmab}
\inf_{\beta\in \Theta}\Big\{\operatorname{tr}\big[b(\beta)v_{z\bar z}\big]-l_{m}(\beta)f\Big\}=0,
\end{equation}
which has the same form of (\ref{equivbell}). Therefore, the equivalence between (\ref{cbellmanmab}) and the complex Hessian equation in (\ref{complexma}) with the relation $v_{z\bar z}\in\overline\Theta$ can be verified by repeating the proof of Lemma \ref{equiveqn}. The theorem is proved.
\end{proof}

\section{Examples of degenerate Hessian equations}\label{secexa}
In this section, we give an account of applications of Theorems \ref{drhp} and \ref{cdrhp} to some celebrated degenerate Hessian equations. Although the examples in Subsections \ref{sb1} $\sim$ \ref{sb4} have been introduced and studied in depth in previous literature, to the author's knowledge, the regularity results for the degenerate cases we state here are new.

\subsection{Two typical ways to construct hyperbolic polynomials}
Given any $m$-th degree homogeneous $A$-hyperbolic polynomial $H_m$ on the real vector space $\mathbb V$, for each $1\le k\le m$, let $\sigma_{k}[H_m;A]$ denote the $k$-th elementary symmetric functions of the $A$-eigenvalues of $H_m$ and $\mu_{k}[H_m;A]$ denote the product of all $k$-sums of the the $A$-eigenvalues of $H_m$, i.e.
$$\sigma_{k}[H_m;A](M)={\sum_{|I|=k}}'\prod_{n=1}^k\lambda_{i_n}[H_m;A](M),$$
$$\mu_{k}[H_m;A](M)={\prod_{|I|=k}}'\sum_{n=1}^{k}\lambda_{i_n}[H_m;A](M).$$
We abbreviate $\sigma_{k}[H_m;A]$ and $\mu_{k}[H_m;A]$ by $\sigma_{k,m}$ and $\mu_{k,m}$ respectively, or simply $\sigma_{k}$ and $\mu_{k}$ when this will not cause confusion.

It is well-known that for each $1\le k\le m$, $\sigma_k$ is a $k$-th degree homogeneous $A$-hyperbolic polynomial, and $\mu_k$ is a $m\choose k$-th degree homogeneous $A$-hyperbolic polynomial. $\sigma_k$ is hyperbolic because by Rolle's theorem the $j$-th order derivative of $p(t)=H_m(M+tA)$ has $(m-j)$ real roots. $\mu_k$ is hyperbolic because $Q(x)=\prod_{|I|=k}'\sum_{n=1}^{k} x_{i_n} $ is a symmetric hyperbolic polynomial with respect to $e=(1,...,1)$ over $\mathbb R^m$, and consequently $q(M)=Q(\lambda(M))$, the composition of $Q$ and the eigenvalue function $\lambda$, where $\lambda(M)=(\lambda_1,...,\lambda_m)(M)$, is hyperbolic with respect to $A$ over $\mathbb V$. (See Section 7 in \cite{MR3055586} for details). It is also not hard to see that their associated G{\aa}rding cones satisfy:
$$\Gamma_{\sigma_1}\supset\Gamma_{\sigma_2}\supset\cdots\supset\Gamma_{\sigma_m},$$
$$\Gamma_{\mu_1}\subset\Gamma_{\mu_2}\subset\cdots\subset\Gamma_{\mu_m}.$$

Now in particular we let $\mathbb V$ be $\mathbb S_+^d$ (in Subsections \ref{sb1} and \ref{sb2}) or $\mathbb H_+^d$ (in Subsections \ref{sb3} and \ref{sb4}), $m=d$, $H=\det_{\mathbb V}$ and $A=I_{\mathbb V}$, so that the eigenvalues are the usual eigenvalues of the matrix. Then for any $C^2$ function on $\Rd$ (resp. $\mathbb C^d$), $1\le k\le d$, the differential operator $\sigma_k(u_{xx})$ (resp. $\sigma_k(u_{z\bar z})$) is called the real (resp. complex) $k$-Hessian operator in previous literature, and we call $\mu_k(u_{xx})$ (resp. $\mu_k(u_{z\bar z})$) the real (resp. complex) $k$-Monge-Amp\`ere operator. It is not hard to see that, say, in the real case,
$$\sigma_1(u_{xx})=\mu_d(u_{xx})=\triangle u,$$
$$\quad \sigma_d(u_{xx})=\mu_1(u_{xx})=\det(u_{xx}),$$
$$\mu_{d-1}(u_{xx})=\prod_{i=1}^d(\triangle u-\lambda_i(u_{xx}))=\det(\triangle u\cdot I-u_{xx}).$$
Since, say, in the real case,
$$\Gamma_{\sigma_d}=\Gamma_{\mu_1}=\mathbb S_+^d,$$
Assumption \ref{ellip} on the ellipticity holds, and Theorem \ref{drhp} is applicable for suitable domains satisfying Assumption \ref{hypbdry}, i.e. the so-called strict $\Gamma$-convexity. It is also worth mentioning that since
\begin{align}
\partial \Gamma_{\sigma_k}\cap\overline{\mathbb S_+^d}=&\{M\in\overline\Gamma: \mbox{$d-k+1$ or more of $\lambda_i(M)$ are zero}\},\label{type1}\\
\partial \Gamma_{\mu_k}\cap\overline{\mathbb S_+^d}=&\{M\in\overline\Gamma: \mbox{$k$ or more of $\lambda_i(M)$ are zero}\},\label{type11}
\end{align}
by (\ref{degenerate}), all of the elliptic operators $\sigma_k$ and $\mu_k$ (restricted on the closure of their associated G{\aa}rding cones) are degenerate except $\sigma_1$ and $\mu_d$.

\subsection{Real $k$-Hessian equations}\label{sb1}
Suppose $d\ge2$ and $2\le k\le d$ are integers. Consider the Dirichlet problem for the degenerate real $k$-Hessian equation:
\begin{equation}\label{sigmak}
\left\{\begin{array}{rcll}
\sigma_k(u_{xx})&=&f^k  &\text{in }D\\ 
u_{xx}&\in& \overline\Gamma_{\sigma_k}&\text{in }D\\
u&=&\varphi  &\text{on }\partial D,
\end{array}
\right.
\end{equation}
where $\Gamma_{\sigma_k}$ is the G{\aa}rding cone associated to $\sigma_k$, $D$ is a bounded smooth domain in $\Rd$ which is strictly $\Gamma_{\sigma_k}$-convex (i.e. satisfying Assumption \ref{hypbdry} for $\Gamma_{\sigma_k}$), and $f: D\rightarrow [0,\infty)$ and $\varphi: \partial D\rightarrow \mathbb R$ are continuous functions.

Assumption \ref{hypbdry} is natural and necessary. By (\ref{type1}) and Lemma A in \cite{MR806416}, $\partial D$ is necessarily connected. (Recall that we suppose $k\ge 2$.) Then in this situation Condition (\ref{curvature}) holds if and only if there exists a sufficiently large real number $t$, such that for all $x\in\partial D$,
$$\sigma_{k,d}\Big(\operatorname{II}_x+t(\operatorname P_{\vec n})_x\Big)>0,$$
which is equivalent to
\begin{equation}\label{k-1convex}
\sigma_{k-1,d-1}(\operatorname{II}_x)>0,\quad\forall x\in\partial D.
\end{equation}
See Remark 1.2 in \cite{MR806416}.
The condition (\ref{k-1convex}) means that at every point on $\partial D$, the $(k-1)$-th elementary symmetric function of its  principle curvatures of $\partial D$ is positive. This was called strict $(k-1)$-convexity in some previous literature. The necessity of these three equivalent conditions was proved in \cite{MR806416}, Proposition 1.3, for vanishing boundary data.

By Lemma \ref{equiveqn}, the system (\ref{sigmak}) is equivalent to
\begin{equation}\label{bellsigmak}
\left\{
\begin{array}{rcll}
\displaystyle\inf_{\alpha\in\Gamma_{\sigma_k}}\big\{a^{ij}(\alpha)u_{x^ix^j}-h(\alpha)f\big\}&=&0&\mbox{ in $D$}\\
u&=&\varphi&\mbox{ on $\partial D$},
\end{array}
\right.
\end{equation}
where
$$a^{ij}(\alpha)=\frac{(\sigma_k)_{\gamma^{ij}}(\alpha)}{\operatorname{tr}\big[(\sigma_k)_{\gamma^{ij}}(\alpha)\big]},\quad\quad h(\alpha)=\frac{k}{d-k+1}\Bigg[\frac{\sigma_k^{1/k}(\alpha)}{\sigma_{k-1}^{1/(k-1)}(\alpha)}\Bigg]^{k-1}.$$

Define the (bounded and Borel) probabilistic solution $u$ to (\ref{bellsigmak}) by (\ref{prma}), i.e. the value function of the stochastic control associated to (\ref{bellsigmak}). Applying Theorem \ref{hyphess}, we have:
\begin{enumerate}
\item If $f,\varphi\in C^{0,1}(\bar D)$, then $u\in C^{0,1}_{loc}(D)\cap C(\bar D)$, and the first (generalized) derivatives satisfy the estimate (\ref{firstma}).
\item If $f,\varphi\in C^{1,1}(\bar D)$, then $u\in C^{1,1}_{loc}(D)\cap C^{0,1}(\bar D)$ and the second (generalized) derivatives satisfy the estimate (\ref{secondma}). Meanwhile, $u$ is the unique strong solution to (\ref{sigmak}), i.e.
\begin{equation*}
\left\{\begin{array}{rcll}
\sigma_k(u_{xx})&=&f^k  &\text{a.e. in }D\\ 
u_{xx}&\in& \overline\Gamma_{\sigma_k}&\text{a.e. in }D\\
u&=&\varphi  &\text{on }\partial D.
\end{array}
\right.
\end{equation*}
\end{enumerate}

\subsection{Real $k$-Monge-Amp\`ere equations}\label{sb2}
Suppose $d\ge 2$ and $1\le k\le d-1$ are integers.
Consider the Dirichlet problem for the degenerate real $k$-Monge-Amp\`ere equation:
\begin{equation}\label{muk}
\left\{\begin{array}{rcll}
\mu_k(u_{xx})&=&f^{d\choose k}  &\text{in }D\\ 
u_{xx}&\in&\overline\Gamma_{\mu_k}&\text{in }D\\
u&=&\varphi  &\text{on }\partial D,
\end{array}
\right.
\end{equation}
where $\Gamma_{\mu_k}$ is the G{\aa}rding cone associated to $\mu_k$, $D$ is a bounded smooth domain in $\Rd$ which is strictly $\Gamma_{\mu_k}$-convex, and $f: D\rightarrow [0,\infty)$ and $\varphi: \partial D\rightarrow \mathbb R$ are continuous functions.

Under assumption \ref{hypbdry}, by (\ref{type11}) and Lemma A in \cite{MR806416}, $\partial D$ is necessarily connected. In this situation, as discussed in \cite{MR2543918}, the condition (\ref{hypbdry}) holds if and only if for each $x\in\partial D$,
$$\operatorname{tr}\big(\operatorname{II}_x|_W\big)>0$$
for all $k$-dimensional hyperplanes $W$ which are tangent to $\partial D$ at $x$,
which is equivalent to
$$\varkappa_1+\cdots+\varkappa_k>0 \mbox{ on }\partial D,$$
where $\varkappa_1\le\cdots\le\varkappa_{d-1}$ are the ordered principle curvatures of $\partial D$ with respect to the interior normal. This condition was also called strict $k$-convexity in some previous literature, which shouldn't be confused with the strict $(k-1)$-convexity mentioned in the previous subsection.

By Lemma \ref{equiveqn}, the system (\ref{muk}) is equivalent to
\begin{equation}\label{bellmuk}
\left\{
\begin{array}{rcll}
\displaystyle\inf_{\alpha\in\Gamma_{\mu_k}}\big\{a^{ij}(\alpha)u_{x^ix^j}-h(\alpha)f\big\}&=&0&\mbox{ in $D$}\\
u&=&\varphi&\mbox{ on $\partial D$},
\end{array}
\right.
\end{equation}
where
$$a^{ij}(\alpha)=\frac{(\mu_k)_{\gamma^{ij}}(\alpha)}{\operatorname{tr}\big[(\mu_k)_{\gamma^{ij}}(\alpha)\big]},\quad\quad h(\alpha)={d\choose k}\frac{\mu_k^{1-1/{d\choose k}}(\alpha)}{(\mu_k)_{(I)}(\alpha)}.$$

Define the probabilistic solution $u$ to (\ref{bellmuk}) by (\ref{prma}). Applying Theorem \ref{hyphess}, we have:
\begin{enumerate}
\item If $f,\varphi\in C^{0,1}(\bar D)$, then $u\in C^{0,1}_{loc}(D)\cap C(\bar D)$, and the first (generalized) derivatives satisfy the estimate (\ref{firstma}).
\item If $f,\varphi\in C^{1,1}(\bar D)$, then $u\in C^{1,1}_{loc}(D)\cap C^{0,1}(\bar D)$ and the second (generalized) derivatives satisfy the estimate (\ref{secondma}). Meanwhile, $u$ is the unique strong solution to (\ref{muk}), i.e.
\begin{equation}\label{strmuk}
\left\{\begin{array}{rcll}
\mu_k(u_{xx})&=&f^{d\choose k}  &\text{a.e. in }D\\ 
u_{xx}&\in& \overline\Gamma_{\mu_k}&\text{a.e. in }D\\
u&=&\varphi  &\text{on }\partial D.
\end{array}
\right.
\end{equation}
\end{enumerate}
\begin{remark}
The second relation in (\ref{strmuk}) means $u$ is $\overline\Gamma_{\mu_k}$-subharmonic in $D$, i.e. the restriction $u|_{D\cap W}$ is subharmonic for all affine $k$-planes $W\subset\Rd$. (cf. Remark \ref{gammasubh} and Section 6, \cite{MR2927376}) 
\end{remark}

\subsection{Complex $k$-Hessian equations}\label{sb3} Suppose $d\ge 2$ and $2\le k\le d$ are integers.
Consider the Dirichlet problem for the degenerate complex $k$-Hessian equation:
\begin{equation}\label{csigmak}
\left\{\begin{array}{rcll}
\sigma_k(u_{z\bar z})&=&f^k  &\text{in }D\\ 
u_{z\bar z}&\in&\overline\Theta_{\sigma_k}&\text{in }D\\ 
u&=&\varphi  &\text{on }\partial D,
\end{array}
\right.
\end{equation}
where $D$ is a bounded smooth domain in $\mathbb C^d$ which is strictly $\Theta_{\sigma_k}$-pseudoconvex (cf. Assumption \ref{chypbdry}), and $f: D\rightarrow [0,\infty)$ and $\varphi: \partial D\rightarrow \mathbb R$ are continuous functions.

In this situation, similarly to the real case, Assumption \ref{chypbdry} holds if and only if 
\begin{equation}\label{k-1pseudoconvex}
\begin{gathered}
\partial D \mbox{ is connected,}\\
\forall x\in\partial D,\quad\sigma_{k-1,d-1}(\operatorname{L}_x)>0.
\end{gathered}
\end{equation}
Recall that $\operatorname L_x$ is the Levi form of $\partial D$ at $x$.

The system (\ref{csigmak}) is equivalent to
\begin{equation}\label{cbellsigmak}
\left\{
\begin{array}{rcll}
\displaystyle\inf_{\beta\in\Theta_{\sigma_k}}\big\{b^{ij}(\beta)u_{z^i\bar z^j}-l(\beta)f\big\}&=&0&\mbox{ in $D$}\\
u&=&\varphi&\mbox{ on $\partial D$},
\end{array}
\right.
\end{equation}
where
$$b^{ij}(\beta)=\frac{(\sigma_k)_{\theta^{ij}}(\beta)}{\operatorname{tr}\big[(\sigma_k)_{\theta^{ij}}(\beta)\big]},\quad\quad l(\alpha)=\frac{k}{d-k+1}\Bigg[\frac{\sigma_k^{1/k}(\beta)}{\sigma_{k-1}^{1/(k-1)}(\beta)}\Bigg]^{k-1}.$$

Defining the probabilistic solution $u$ to (\ref{cbellsigmak}) by (\ref{cprma}) and applying Theorem \ref{chyphess}, we obtain:
\begin{enumerate}
\item If $f,\varphi\in C^{0,1}(\bar D)$, then $u\in C^{0,1}_{loc}(D)\cap C(\bar D)$, and the first (generalized) derivatives satisfy the estimate (\ref{cfirstma}).
\item If $f,\varphi\in C^{1,1}(\bar D)$, then $u\in C^{1,1}_{loc}(D)\cap C^{0,1}(\bar D)$ and the second (generalized) derivatives satisfy the estimate (\ref{csecondma}). Meanwhile, $u$ is the unique strong solution to (\ref{csigmak}).
\end{enumerate}

\subsection{Complex $k$-Monge-Amp\`ere equations}\label{sb4}
Suppose $d\ge 2$ and $1\le k\le d-1$ are integers. Consider the Dirichlet problem for the degenerate complex $k$-Monge-Amp\`ere equation:
\begin{equation}\label{cmuk}
\left\{\begin{array}{rcll}
\mu_k(u_{z\bar z})&=&f^{m\choose k}  &\text{in }D\\ 
u_{z\bar z}&\in&\overline\Theta_{\mu_k}&\text{in }D\\
u&=&\varphi  &\text{on }\partial D.
\end{array}
\right.
\end{equation}
where $\Theta_{\mu_k}$ is the G{\aa}rding cone associated to $\mu_k$, $D$ is a bounded smooth domain in $\mathbb C^d$ which is strictly $\Theta_{\mu_k}$-pseudoconvex, and $f: D\rightarrow [0,\infty)$ and $\varphi: \partial D\rightarrow \mathbb R$ are continuous functions. 

In this situation, Assumption \ref{chypbdry} holds if and only if 
\begin{equation}\label{kpseudoconvex}
\begin{gathered}
\partial D \mbox{ is connected,}\\
\varkappa_1+\cdots+\varkappa_k>0 \mbox{ on }\partial D,
\end{gathered}
\end{equation}
where $\varkappa_1\le\cdots\le\varkappa_{d-1}$ are the ordered (real) eigenvalues of the Levi form of $\partial D$.

The system (\ref{cmuk}) is equivalent to
\begin{equation}\label{cbellmuk}
\left\{
\begin{array}{rcll}
\displaystyle\inf_{\beta\in\Theta_{\mu_k}}\big\{b^{ij}(\beta)u_{z^i\bar z^j}-l(\beta)f\big\}&=&0&\mbox{ in $D$}\\
u&=&\varphi&\mbox{ on $\partial D$},
\end{array}
\right.
\end{equation}
where
$$b^{ij}(\beta)=\frac{(\mu_k)_{\theta^{ij}}(\beta)}{\operatorname{tr}\big[(\mu_k)_{\theta^{ij}}(\beta)\big]},\quad\quad l(\beta)={d\choose k}\frac{\mu_k^{1-1/{d\choose k}}(\beta)}{(\mu_k)_{(I)}(\beta)}.$$

Defining the probabilistic solution $u$ to (\ref{cbellmuk}) by (\ref{cprma}) and applying Theorem \ref{chyphess}, we obtain:
\begin{enumerate}
\item If $f,\varphi\in C^{0,1}(\bar D)$, then $u\in C^{0,1}_{loc}(D)\cap C(\bar D)$, and the first (generalized) derivatives satisfy the estimate (\ref{cfirstma}).
\item If $f,\varphi\in C^{1,1}(\bar D)$, then $u\in C^{1,1}_{loc}(D)\cap C^{0,1}(\bar D)$ and the second (generalized) derivatives satisfy the estimate (\ref{csecondma}). Meanwhile, $u$ is the unique strong solution to (\ref{cmuk}).
\end{enumerate}

\section{Real Hessian equations under general settings}\label{section5}

In this section, we generalize the results we obtained in Section \ref{section2}. Roughly speaking, by going through the proof in Section \ref{section2}, we notice that the results are still true when replacing $H_m^{1/m}$ with any monotonic, $1$-homogeneous, concave, symmetric function defined on any cone on $\mathbb S^d$ which is elliptic, convex and symmetric. 

\subsection{Statement of the results} We consider the Dirichlet problem for the general real Hessian equation:
\begin{equation}\label{rhess}
\left\{\begin{array}{rcll}
F(u_{xx})&=&f  &\text{in }D\\ 
u_{xx}&\in&\overline\Gamma&\text{in }D\\
u&=&\varphi  &\text{on }\partial D,
\end{array}
\right.
\end{equation}
where $D$ is a bounded smooth domain in $\Rd$ ($d\ge2$), $f: D\rightarrow [0,\infty)$ and $\varphi: \partial D\rightarrow \mathbb R$ are bounded and Borel measurable functions, and $\Gamma$, $F$ and $D$ satisfy the following Assumptions \ref{a1}, \ref{a2} and \ref{a3}, respectively.

Our assumptions here are the following. 

\begin{assumption}\label{a1} $\Gamma$ is a nonempty proper open cone in $\mathbb S^d$ with vertex at the origin, satisfying the following conditions:
\begin{enumerate}
\item (Ellipticity) $\overline\Gamma+\overline{\mathbb S^d_+}\subset\overline\Gamma$.
\item (Convexity) $\Gamma$ is convex.
\item (Symmetry) $O \Gamma O^*=\Gamma,\forall O\in\mathbb O^d$.
\end{enumerate}
\end{assumption}

\begin{assumption} \label{a2}  $F$ is a function defined on $\overline{\Gamma}$ satisfying 
\begin{equation}\label{match}
F|_{\partial \Gamma}=0,\quad F|_{\Gamma}>0
\end{equation} 
and the following conditions:
\begin{enumerate}
\item (Homogeneity) $F(tM)=tF(M), \forall t\ne0, M\in \overline\Gamma$.
\item (Concavity) $F$ is concave.
\item (Symmetry) $F(OMO^*)=F(M), \forall M\in \Gamma,O\in\mathbb O^d$.
\end{enumerate}
\end{assumption}

\begin{assumption}\label{a3} $D$ is
strictly $\Gamma$-convex. (cf. Assumption \ref{hypbdry})
\end{assumption}

This setup generally covers the situations studied in Section \ref{section2} and includes more interesting examples to be discussed in Section \ref{secexamore}.

\begin{remark}\label{rmk50}
With Assumptions \ref{a1} and \ref{a2}, the $\overline{\mathbb S_+^d}$-monotonicity of $F$, i.e. the condition
\begin{equation}\label{mnt}
F(M+S)\ge F(M),\quad \forall M\in\Gamma, S\in\overline{\mathbb S_+^d},
\end{equation}
is automatically true, by either (1) and (2) in Assumption \ref{a2} and then Assumption \ref{a1} (1), or Condition (\ref{match}), Assumption \ref{a2} (2) and Assumption \ref{a1} (1). 
\end{remark}

\begin{remark}\label{rmk51}
The assumption on the homogeneity of $F$ is not necessarily of order 1. To be precise, suppose $F$ satisfies Assumption \ref{a2} but with $m$-homogeneity,  $m\ne1$, rather than $1$-homogeneity. If $m>1$, then $F^{1/m}$ satisfies Assumption \ref{a2}. (Because $\phi(t)=t^{1/m}$ is increasing and concave on $[0,\infty)$ when $m>1$.) If $0<m<1$, then $F^{1/m}$ satisfies Assumption \ref{a2} provided that $F^{1/m}$ is concave.

\end{remark}

\begin{remark}\label{rmk52}
The concavity of $F$ implies the local Lipschitz continuity of $F$, which is all we need on the regularity of it.
\end{remark}

\begin{remark}\label{rmk53}
The homogeneity condition on $F$, i.e. Assumption \ref{a2} (1) can be weakened by some growth condition of $F$ in the direction of $I$. Precisely, Assumption \ref{a2} (1) can be replaced with
\begin{enumerate}
\item[(1')] 
$F_{(I)}(M)\ge\nu$ at each $M\in\Gamma$ where $F$ is differentiable in the direction of $I$, with $\nu$ a  positive constant.
\end{enumerate}
or
\begin{enumerate}
\item[(1'')] 
$\displaystyle\lim_{t\rightarrow+\infty}F(M+tI)=+\infty, \forall M\in\partial \Gamma.$
\end{enumerate}
Naturally, when either (1') or (1'') was adopted, the set $\Gamma$ is still supposed to satisfy Assumption \ref{a1} except that it is a cone. 

This can be understood after going through the proof of Theorem \ref{rhess}, as we will explain in Remark \ref{explain}, given right after we finish the proof of Theorem \ref{rhess}.
\end{remark}

To introduce the stochastic representation of the solution to (\ref{rhess}), we first define 
$$\Gamma'=\{\gamma\in\Gamma: F \mbox{ is differentiable at }\gamma\}$$
and on $\Gamma'$
\begin{equation}\label{gaalpha}
a^{ij}(\gamma)=\frac{F_{\gamma^{ij}}(\gamma)}{\operatorname{tr}\big[F_{\gamma^{ij}}(\gamma)\big]},\quad h(\gamma)=\frac{1}{F_{(I)}(\gamma)}.
\end{equation}
(We will show in Lemma \ref{lm1} that on $\Gamma'$, the symmetric matrix $a(\gamma):=(a^{ij}(\gamma))_{1\le i,j\le d}$ is semi-positive and $\operatorname{tr}[F_{\gamma^{ij}}(\gamma)]=F_{(I)}(\gamma)$ is strictly positive, so (\ref{gaalpha}) is well-defined.)

Then we let $w_t$ be a Wiener process of dimension $d$ and $\mathfrak A$ be the set of progressively-measurable processes $\alpha=(\alpha_t)_{t\ge0}$ with values in $\Gamma'$ for all $t\ge0$. On $\bar D$, define
\begin{equation}\label{grp}
v(x)=\inf_{\alpha\in\mathfrak A}E\bigg[\varphi\big(x^{\alpha,x}_{\tau^{\alpha,x}}\big)-\int_0^{\tau^{\alpha,x}}h(\alpha_t)f\big(x_t^{\alpha,x}\big)dt\bigg],
\end{equation}
with
\begin{equation}\label{grp2}
x_t^{\alpha,x}=x+\int_0^t\sqrt{2a(\alpha_s)}dw_s,  \ \forall \alpha\in\mathfrak A,
\end{equation}
and 
$$\tau^{\alpha,x}=\inf\{t\ge0:x_t^{\alpha,x}\notin D\}.$$

\begin{theorem}\label{ghessthm}
$v$ given by (\ref{grp}) is Borel measurable and bounded, and it satisfies the estimate (\ref{0ma}) with $mH_m^{-1/m}(I)$ replaced by $1/F(I)$.

 If $f,\varphi\in C^{0,1}(\bar D)$, then $v\in C^{0,1}_{loc}(D)\cap C(\bar D)$, and the first (generalized) derivatives satisfy the estimate (\ref{firstma}).

 If $\varphi\in C^{1,1}(\bar D)$, $f\in C^{0,1}(\bar D)$ and is quasi-convex on $\bar D$, then $v\in C^{1,1}_{loc}(D)\cap C^{0,1}(\bar D)$ and the second (generalized) derivatives satisfy the estimate (\ref{secondma}). Meanwhile, $v$ is the unique strong solution to (\ref{rhess}), i.e.
 \begin{equation}\label{rhessae}
\left\{\begin{array}{rcll}
F(u_{xx})&=&f  &\text{a.e. in }D\\ 
u_{xx}&\in&\overline\Gamma&\text{a.e. in }D\\
u&=&\varphi  &\text{on }\partial D.
\end{array}
\right.
\end{equation}

\end{theorem}

\begin{remark}
Recall Remarks \ref{viscosity} and \ref{gammasubh} to interested readers.
\end{remark}

\subsection{Proof of Theorem {\ref{ghessthm}}}
Similarly to Section \ref{section2}, we prove Thereom \ref{ghessthm} by first rewriting the system (\ref{rhess}) as a Bellman equation with Dirichlet boundary condition and then apply the regularity results obtained in \cite{InteriorRegularityI}.

\begin{lemma}\label{lm1}
On $\Gamma'$, Define $T(\alpha)=\big[T^{ij}(\alpha)\big]_{d\times d}$, where $T^{ij}(\alpha)=F_{\gamma^{ij}}(\alpha)$. Then for each $\alpha\in\Gamma'$, we have 
\begin{enumerate}
\item$T(\alpha)\in  \overline{\mathbb S^d_+}$;
\item $\operatorname{tr}T(\alpha)=F_{(I)}(\alpha)\ge F(I)>0$;
\item $O\alpha O^*\in\Gamma'$ and $T(O\alpha O^*)=OT(\alpha)O^*$, $\forall O\in \mathbb O^d$.

\end{enumerate}

\end{lemma}
\begin{proof}
For any $\alpha\in\Gamma'$ and $S\in\overline{\mathbb S_+^d}$, by Assumption \ref{a2} (1) and (2),
$$\frac{F(\alpha+tS)-F(\alpha)}{t}\ge F(S).$$
Letting $S=\zeta\zeta^*$, for all $\zeta\in\Rd$, we obtain (1). When $S=I$, we obtain (2).

By Assumption \ref{a2} (3), we have $O\alpha O^*\in \Gamma'$ for each $\alpha\in\Gamma'$. Repeating (\ref{repeat}) for $F$ with $M\in\overline{\mathbb S_+^d}$, we obtain (3).
\end{proof}

Due to Lemma \ref{lm1}, $a^{ij}$ and $h$ in (\ref{gaalpha}) are well-defined on $\Gamma'$ and for each $\alpha\in\Gamma'$, $a(\alpha)\in \overline{\mathbb S_+^d}$ and $\tr a(\alpha)=1$.

\begin{lemma}\label{lm2}
For any $c\ge0$, the system
\begin{equation}\label{gequivhyp}
\left\{\begin{array}{rcl}
F(\gamma)&=&c\\
\gamma&\in&\overline\Gamma
\end{array}
\right.
\end{equation}
is equivalent to the  Bellman equation
\begin{equation}\label{gequivbell}
\inf_{\alpha\in\Gamma'}\big\{a^{ij}(\alpha)\gamma_{ij}-h(\alpha)c\big\}=0,
\end{equation}
where $a(\alpha)$ and $h(\alpha)$ are define in (\ref{gaalpha}).

\end{lemma}

\begin{proof}[Proof of (\ref{gequivhyp})$\Rightarrow$(\ref{gequivbell})] By Assumption \ref{a2} (1) and (2), on $\Gamma$,
\begin{equation}\label{Fgamma}
F(\gamma)=\inf_{\alpha\in\Gamma'}\big\{F_{\gamma^{ij}}(\alpha)\gamma_{ij}\big\},
\end{equation}
which implies that (\ref{gequivhyp}) is equivalent to 
$$\inf_{\alpha\in\Gamma'}\big\{F_{\gamma^{ij}}(\alpha)\gamma_{ij}-c\big\}=0.$$

For each $\gamma\in\Gamma'$, the infimum of the expression in the bracket is attained at $\alpha=\gamma$, so we obtain
$$F_{\gamma^{ij}}(\gamma)\gamma_{ij}=c.$$
By Lemma \ref{T} (3), for each $\gamma\in\Gamma$, we can divide both sides by $\operatorname{tr}T(\gamma)$,  so
\begin{align*}
a^{ij}(\gamma)\gamma_{ij}=h(\gamma)c.
\end{align*}
Therefore for each $\gamma\in\Gamma'$, Equation (\ref{gequivbell}) is true since again the infimum there is attained at $\alpha=\gamma$.

For each $\gamma\in\overline\Gamma\setminus\Gamma'$, there exists a sequence $\{\epsilon_n\}_{n=1}^{\infty}$, satisfying $\epsilon_n>0$, $\lim_{n\rightarrow\infty}\epsilon_n=0$ and $\gamma+\epsilon_nI\in\Gamma'$ for each $n$. Therefore,
\begin{align*}
a^{ij}(\gamma+\epsilon_n I)\gamma_{ij}-h(\gamma+\epsilon_n I)c=&\frac{F_{\gamma^{ij}}(\gamma+\epsilon_n I)\gamma_{ij}-c}{\operatorname{tr}\big[F_{\gamma^{ij}}(\gamma+\epsilon_n I)\big]}\\
=&\frac{F_{\gamma^{ij}}(\gamma+\epsilon_n I)(\gamma_{ij}+\epsilon_n\delta_{ij}-\epsilon_n\delta_{ij})-c}{\operatorname{tr}\big[F_{\gamma^{ij}}(\gamma+\epsilon_n I)\big]}\\
=&\frac{F(\gamma+\epsilon_n I)-c}{\operatorname{tr}\big[F_{\gamma^{ij}}(\gamma+\epsilon_n I)\big]}-\epsilon_n.
\end{align*}
By Lemma \ref{lm1} (2), we obtain
$$\lim_{n\rightarrow\infty}\big[a^{ij}(\gamma+\epsilon_n I)\gamma_{ij}-h(\gamma+\epsilon_n I)c\big]=0,$$
which implies that Equation (\ref{gequivbell}) is true on $\overline\Gamma\setminus\Gamma'$.

\end{proof}

\begin{proof}[Proof of (\ref{gequivbell})$\Rightarrow$(\ref{gequivhyp})]
It is almost the same as the proof of (\ref{equivbell})$\Rightarrow$(\ref{equivhyp}) in Section \ref{section2}. 

We first establish the relation $\gamma\in\overline\Gamma$. By (\ref{Fgamma}),
\begin{align*}
\overline\Gamma=&\{\gamma\in\mathbb S^d: F_{\gamma^{ij}}(\alpha)\gamma_{ij}\ge0,\forall \alpha\in\Gamma'\}\\
=&\{\gamma\in\mathbb S^d: a^{ij}(\alpha)\gamma_{ij}\ge0,\forall \alpha\in\Gamma'\}.
\end{align*}
On the other hand, since $c\ge0$, $h(\alpha)>0$ on $\Gamma'$, given any $\gamma$ satisfying (\ref{gequivbell}), we have
$$a^{ij}(\alpha)\gamma_{ij}\ge0, \quad\forall\alpha\in\Gamma'.$$
Thus the relation is established.

Then we verify the equation $F(\gamma)=c$. Given any $\gamma\in\overline\Gamma$, we notice that 
$$\lim_{t\rightarrow+\infty}F(tI+\gamma)=\lim_{t\rightarrow+\infty}tF(I+t^{-1}\gamma)=+\infty,$$
and that there exists $t_0\le0$, such that $t_0I+\gamma\in\partial\Gamma$ and $F(t_0I+\gamma)=0$. By the continuity of $\phi(t)=F(tI+\gamma)$, for each $c\ge0$, there exists a scalar $t_1\in [t_0,\infty)$, such that $t_1I+\gamma\in\overline\Gamma$ and $F(t_1I+\gamma)=c$.

From the proof of (\ref{gequivhyp})$\Rightarrow$(\ref{gequivbell}), we know that
\begin{align*}
\inf_{\alpha\in\Gamma'}\Big\{a^{ij}(\alpha)(t_1\delta_{ij}+\gamma_{ij})-h(\alpha)c\Big\}=&0\\
\inf_{\alpha\in\Gamma'}\Big\{a^{ij}(\alpha)\gamma_{ij}-h(\alpha)c\Big\}=&-t_1.
\end{align*}
Comparison with (\ref{gequivbell}) gives $t_1=0$, so Equation (\ref{gequivhyp}) holds.
\end{proof}

\begin{lemma}\label{gpsi}
For sufficiently large constant $C$, the global barrier $\psi$ given in Definition \ref{globalbarrier} satisfies the following conditions:
\begin{enumerate}
\item $D=\{\psi>0\}$;
\item $|\psi_x|\ge 1$ on $\partial D$;
\item $\displaystyle\sup_{\alpha\in\Gamma'}\big(a_{ij}(\alpha)\psi_{x^ix^j}\big)\le-1$ in $\bar D$.
\end{enumerate}
\end{lemma}
\begin{proof}
Repeat the proof of Lemma \ref{psi} with $\Gamma$ replaced by $\Gamma'$.
\end{proof}

With the above three lemmas, we are ready to prove Theorem \ref{ghessthm}.

\begin{proof}[Proof of Theorem \ref{ghessthm}]
We apply Theorems 2.1, 2,2 and 2.3 in \cite{InteriorRegularityI} with
$$A=\Gamma',\quad a^\alpha=\Bigg(\frac{F_{\gamma^{ij}}(\alpha)}{\operatorname{tr}\big[F_{\gamma^{ij}}(\alpha)\big]}\Bigg)_{1\le i,j\le d},\quad \sigma^\alpha=\sqrt{2a^\alpha},$$
$$b^\alpha=0,\quad c^\alpha=0,\quad f^\alpha(x)=\frac{f(x)}{F_{(I)}(\alpha)},\quad g(x)=-\varphi(x).$$

Based on these substitutions, it is not hard to see that $v$ defined by (\ref{grp}) equals $-v$ defined by (2.3) in \cite{InteriorRegularityI}.

By Lemma \ref{gpsi}, Assumption 2.1 in \cite{InteriorRegularityII} holds. By Lemma \ref{lm1} (3), Assumption 2.2 in \cite{InteriorRegularityI} holds. Since $I\in \Gamma$, the weak non-degeneracy condition holds. Therefore we can apply Theorems 2.1, 2.2 and 2.3 in \cite{InteriorRegularityI} to obtain Theorem \ref{ghessthm}.

\end{proof}

\begin{remark}\label{explain} We explain here Remark \ref{rmk53}. Throughout the proof of Theorem \ref{ghessthm} in this subsection, we notice that the properties of $F$ needed in the proof and derived from the homogeneity condition are the following.
\begin{enumerate}
\item[(i)]$F_{(I)}(M)\ge\nu$ at each $M\in\Gamma$ where $F$ is differentiable in the direction of $I$, with $\nu$ a  positive constant;
\item[(ii)]$\displaystyle\lim_{t\rightarrow+\infty}F(M+tI)=+\infty, \forall M\in\partial \Gamma.$
\end{enumerate}

Observe that (i) implies (ii), therefore, Condition (1') in Remark \ref{rmk53} can substitute Condition (1) in Assumption \ref{a2}. In this situation, the Bellman equation equivalent to (\ref{gequivhyp}) is
$$\inf_{\alpha\in\Gamma'}\big\{\hat a^{ij}(\alpha)\gamma_{ij}- \hat h(\alpha)(F_{(\alpha)}(\alpha)-F(\alpha)+c)\big\}=0,$$
where
\begin{align*}
\hat a^{ij}(\alpha)=&\frac{F_{\gamma^{ij}}(\alpha)}{\sqrt{(F_{(I)}(\alpha))^2+(F_{(\alpha)}(\alpha)-F(\alpha))^2}},\\
\hat h(\alpha)=&\frac{1}{\sqrt{(F_{(I)}(\alpha))^2+(F_{(\alpha)}(\alpha)-F(\alpha))^2}}.
\end{align*}
Note that if $F$ is $1$-homogeneous, then this Bellman equation reduces to (\ref{gequivbell}).

On the other hand, $\overline{\mathbb S_+^d}$-monotonicity of $F$ and (ii), along with concavity of $F$ imply that $F_{(I)}(M)>0$ at each $M\in\Gamma$ where $F$ is differentiable in the direction of $I$. Indeed, if there exists $M\in\Gamma$ at which $F_{(I)}(M)=0$, then by $\overline{\mathbb S_+^d}$-monotonicity and concavity of $F$, $F(M+tI)\equiv F(M)$ for all $t\ge0$, which is a contradiction to (ii). Then in this situation, we can normalize the Bellman equation equivalent to (\ref{gequivhyp}) as
$$\inf_{\alpha\in\Gamma'}\big\{\hat a^{ij}(\alpha,c)\gamma_{ij}-\hat h(\alpha,c)\big\}=0,$$
where
\begin{align*}
\hat a^{ij}(\alpha,c)=&\frac{F_{\gamma^{ij}}(\alpha)}{\sqrt{(F_{(I)}(\alpha))^2+(F_{(\alpha)}(\alpha)-F(\alpha)+c)^2}},\\
\hat h(\alpha,c)=&\frac{F_{(\alpha)}(\alpha)-F(\alpha)+c}{\sqrt{(F_{(I)}(\alpha))^2+(F_{(\alpha)}(\alpha)-F(\alpha)+c)^2}}.
\end{align*}
Therefore, Condition (1'') in Remark \ref{rmk53} can substitute Condition (1) in Assumption \ref{a2}.
\end{remark}

\section{Complex Hessian equations under general settings}\label{section6}
This section is the complex counterpart of Section \ref{section5}. 

\subsection{Statement of the results}
Our setup here is the following. We consider the Dirichlet problem for the general complex Hessian equation:
\begin{equation}\label{chess}
\left\{\begin{array}{rcll}
G(u_{z\bar z})&=&f  &\text{in }D\\ 
u_{z\bar z}&\in&\overline\Theta&\text{in }D\\
u&=&\varphi  &\text{on }\partial D,
\end{array}
\right.
\end{equation}
where $D$ is a bounded smooth domain in $\mathbb C^d$ ($d\ge2$), $f: D\rightarrow [0,\infty)$ and $\varphi: \partial D\rightarrow \mathbb R$ are bounded and Borel measurable functions, and $\Theta$, $G$ and $D$ satisfy the following Assumptions \ref{ca1}, \ref{ca2} and \ref{ca3}, respectively.

\begin{assumption}\label{ca1} $\Theta$ is a nonempty proper open cone in $\mathbb H^d$ with vertex at the origin, satisfying the following conditions:
\begin{enumerate}
\item (Ellipticity) $\overline\Theta+\overline{\mathbb H^d_+}\subset\overline\Theta$.
\item (Convexity) $\Theta$ is convex.
\item (Symmetry) $U \Theta \bar U^*=\Theta,\forall U\in\mathbb U^d$.
\end{enumerate}
\end{assumption}

\begin{assumption} \label{ca2}  $G$ is a real-valued function defined on $\overline{\Theta}$ satisfying 
$$G|_{\partial \Theta}=0,\quad G|_{\Theta}>0$$ 
and the following conditions:
\begin{enumerate}
\item (Homogeneity) $G(tM)=tG(M), \forall t\ne0, M\in \overline\Theta$.
\item (Concavity) $G$ is concave.
\item (Symmetry) $G(UM\bar U^*)=G(M), \forall M\in \Theta,U\in\mathbb U^d$.
\end{enumerate}
\end{assumption}

\begin{assumption}\label{ca3} $D$ is
strictly $\Theta$-pseudoconvex. (cf. Assumption \ref{chypbdry})
\end{assumption}

This setting generally covers the problem studied in Section \ref{section3} and includes more interesting examples to be discussed in Section \ref{secexamore}.

\begin{remark}
Remarks similar to Remarks \ref{rmk50} $\sim$ \ref{rmk53} can also be made in complex cases.
\end{remark}

To introduce the stochastic representation of the solution to (\ref{chess}), we first define 
$$\Theta'=\{\theta\in\Theta: G \mbox{ is differentiable at }\theta\}$$
and on $\Theta'$
\begin{equation}\label{cgaalpha}
b^{ij}(\theta)=\frac{G_{\theta^{ij}}(\theta)}{\operatorname{tr}\big[G_{\theta^{ij}}(\theta)\big]},\quad l(\theta)=\frac{1}{G_{(I)}(\theta)}.
\end{equation}
(By an argument similar to that in Lemma \ref{lm1}, we know that on $\Theta'$, the hermitian matrix $b(\theta):=(b^{ij}(\theta))_{1\le i,j\le d}$ is semi-positive and $\operatorname{tr}[G_{\theta^{ij}}(\theta)]=1/G_{(I)}(\theta)$ is strictly positive, so (\ref{cgaalpha}) is well-defined.)

Then we let $W_t$ be the normalize complex Wiener process of dimension $d$ and $\mathfrak B$ be the set of progressively-measurable processes $\beta=(\beta_t)_{t\ge0}$ with values in $\Theta$ for all $t\ge0$. On $\bar D$, define
\begin{equation}\label{cgrp}
v(z)=\inf_{\beta\in\mathfrak B}E\bigg[\varphi\big(z^{\beta,z}_{\tau^{\beta,z}}\big)-\int_0^{\tau^{\beta,z}}l(\beta_t)f\big(z_t^{\beta,z}\big)dt\bigg].
\end{equation}
with
\begin{equation*}
z_t^{\beta,z}=z+\int_0^t\sqrt{b(\beta_s)}dW_s,  \ \forall \beta\in\mathfrak B,
\end{equation*}
and 
$$\tau^{\beta,z}=\inf\{t\ge0:z_t^{\beta,z}\notin D\}.$$

\begin{theorem}\label{cghessthm}
$v$ given by (\ref{cgrp}) is Borel measurable and bounded, and it satisfies the estimate (\ref{c0ma}) with $mG_m^{-1/m}(I)$ replaced by $1/G(I)$.

 If $f,\varphi\in C^{0,1}(\bar D)$, then $v\in C^{0,1}_{loc}(D)\cap C(\bar D)$, and the first (generalized) derivatives satisfy the estimate (\ref{cfirstma}).

 If $f,\varphi\in C^{1,1}(\bar D)$, then $v\in C^{1,1}_{loc}(D)\cap C^{0,1}(\bar D)$ and the second (generalized) derivatives satisfy the estimate (\ref{csecondma}). Meanwhile, $v$ is the unique strong solution to (\ref{chess}), i.e.
 \begin{equation}\label{chessae}
\left\{\begin{array}{rcll}
G(u_{z\bar z})&=&f  &\text{a.e. in }D\\ 
u_{z\bar z}&\in&\overline\Theta&\text{a.e. in }D\\
u&=&\varphi  &\text{on }\partial D.
\end{array}
\right.
\end{equation}

\end{theorem}

\subsection{Proof of Theorem {\ref{cghessthm}}} We still follow the three-step procedure in Subsection \ref{subsection32}.

\begin{proof}[Proof of Theorem \ref{cghessthm}]

To rewrite the payoff function in (\ref{cgrp}) as a function on $\mathbb R^{2d\times 2d}$, we first notice that
$$\Phi( z_t^{\beta,z})=\Phi z+\int_0^t\frac{1}{\sqrt 2}\Big(\Phi\sqrt{b({\beta_s})}\Big)dw_s=\Phi z+\int_0^t\frac{1}{\sqrt 2}\Big(\sqrt{\Phi b({\beta_s})}\Big)dw_s,$$
where the homomorphism $\Phi$ was defined at the beginning of the proof of Theorem \ref{cdrhp} and $w_t$ is a $2d$-dimensional Wiener process. Let
$$\alpha=\Phi\beta, \qquad\mathfrak A=\Phi\mathfrak B=\{\Phi\beta:\beta\in\mathfrak B\},$$
$$\Phi\Theta'=\{\Phi\beta:\beta\in\Theta'\}.$$
On $\Phi\Theta'$, define $F(\alpha)=G(\Phi^{-1}\alpha)$. Then for each $\beta\in\Theta'$ and $1\le i,j\le d$, we have
\begin{align*}
G_{\theta^{ij}}(\beta)=&\frac{\partial}{\partial\theta^{ij}}F(\Phi\beta)\\
=&F_{\gamma^{kl}}(\Phi\beta)\frac{\partial}{\partial\theta^{ij}}(\Phi\beta)^{kl}\\
=&\frac{1}{2}F_{\gamma^{ij}}(\Phi\beta)+\frac{1}{2}F_{\gamma^{i+d,j+d}}(\Phi\beta)+\frac{i}{2}F_{\gamma^{i,j+d}}(\Phi\beta)-\frac{i}{2}F_{\gamma^{i+d,j}}(\Phi\beta),
\end{align*}
Since $(F_{\gamma^{ij}}(\Phi\beta))_{1\le i,j\le 2d}\in \Phi\mathbb H^d$, 
$$F_{\gamma^{ij}}(\Phi\beta)=F_{\gamma^{i+d,j+d}}(\Phi\beta),\quad F_{\gamma^{i,j+d}}(\Phi\beta)=-F_{\gamma^{i+d,j}}(\Phi\beta).$$
Therefore, we obtain that for each $\beta\in\Theta'$,
$$\Phi G_{\theta^{ij}}(\beta)=F_{\gamma^{ij}}(\alpha),$$
which implies that
$$\Phi(b(\beta))=\bigg(\frac{2F_{\gamma^{ij}}(\alpha)}{\operatorname{tr}\big[F_{\gamma^{ij}}(\alpha)\big]}\bigg)_{1\le i,j\le 2d}:=4a(\alpha).$$
Therefore we can rewrite (\ref{cgrp}) as
\begin{equation}\label{vrealized}
v(x)=\inf_{\alpha\in\mathfrak A}E\bigg[\varphi(x^{\alpha,x}_{\tau^{\alpha,x}})-\int_0^{\tau^{\alpha,x}}f^{\alpha_t}(x^{\alpha,x}_t)dt\bigg],
\end{equation}
where
$$x_t^{\alpha,x}=x+\int_0^t\sqrt{2a(\alpha_s)}dw_s,$$
$$f^{\alpha}(x)=\frac{2f(x)}{F_{(I)}(\alpha)}.$$
Then we repeat Step 2 in the proof of Theorem \ref{cdrhp} to obtain all regularity results on $v(x)$ defined by (\ref{vrealized}). 

It remains to verify that the associated dynamic programing equation is equivalent to the complex Hessian equation. Firstly, the associated dynamic programing equation of $v(x)$ is the real Bellman equation
\begin{equation*}
\inf_{\alpha\in \Phi\Theta}\Big\{\operatorname{tr}\big[a(\alpha)v_{xx}\big]-f^{\alpha}(x)\Big\}=0,
\end{equation*}
which is equivalent to 
\begin{equation*}
\inf_{\beta\in \Theta}\Big\{\frac{1}{4}\operatorname{tr}\big[(\Phi b(\beta))v_{xx})\big]-l(\beta)f\Big\}=0.
\end{equation*}
To write down the corresponding complex Bellman equation, it suffices to notice that
$$\operatorname{tr}\big[(\Phi b(\beta))v_{xx}\big]=4\operatorname{tr}\big[b(\beta)v_{z\bar z}\big].$$
Therefore the complex Bellman equation is
\begin{equation}\label{cgbellmanmab}
\inf_{\beta\in \Theta}\Big\{\operatorname{tr}\big[b(\beta)v_{z\bar z}\big]-l(\beta)f\Big\}=0.
\end{equation}
Lastly, the equivalence between (\ref{cgbellmanmab}) and the complex Hessian equation in (\ref{chess}) with the relation $v_{z\bar z}\in\overline\Theta$ can be verified by repeating the proof of Lemma \ref{lm2}. 
\end{proof}

\section{Further examples of degenerate Hessian equations}\label{secexamore}
The point of this section is to show that there are many  applications of the main results obtained in Sections \ref{section5} and \ref{section6}. We in particular study the representation and regularity for the  Dirichlet problem for three interesting types of degenerate Hessian equations.

Throughout this section $D$ is supposed to be a bounded smooth domain (in $\Rd$ or $\mathbb C^d$, with $d\ge2$), $f: D\rightarrow [0,\infty)$ and $\varphi: \partial D\rightarrow \mathbb R$ are assumed to be bounded and Borel measurable functions.

\subsection{Minimal/Maximal eigenvalue equations}
Given an $m$-th degree hyperbolic polynomial $H_m$ over $\mathbb S^d$ with respect to $I$, we denote by $\lambda_{\min}[H_m](M)$ (resp. $\lambda_{\max}[H_m](M)$) the minimal (resp. maximal) eigenvalue of $H_m(M)$. It well-known that $\lambda_{\min}[H_m]$ is superlinear, i.e. concave and $1$-homogeneous. (See Theorem 2, \cite{MR0113978} or Proposition 3.4, \cite{MR3055586}). 
Therefore we can apply Theorem \ref{rhess} with $F=\lambda_{\min}[H_m]$ and $\Gamma=\Gamma_{H_m}$ to study the Dirichlet problem
\begin{equation}\label{mineigen}
\left\{\begin{array}{rcll}
\lambda_{\min}[H_m](u_{xx})&=&f  &\text{in }D\\ 
u_{xx}&\in&\overline\Gamma_{H_m}&\text{in }D\\
u&=&\varphi  &\text{on }\partial D,
\end{array}
\right.
\end{equation}
if we suppose that $H_m$ is symmetric (cf. Assumption \ref{a2} (4)), $\Gamma_{H_m}$ is elliptic (cf. Assumption \ref{a1} (1)) and $D$ is strictly $\Gamma_{H_m}$-convex (cf. Assumption \ref{a3}).
\begin{example}
In particular, $\lambda_{\min}[\det](u_{xx})$ is the smallest usual eigenvalue of the Hessian matrix of $u$, and for $1\le k\le d$, $\lambda_{\min}[\mu_k](u_{xx})$ is the arithmetic mean of the first $k$ smallest eigenvalues of $u_{xx}$, i.e.
$$\lambda_{\min}[\mu_k](u_{xx})=\frac{\lambda_1(u_{xx})+\cdots+\lambda_k(u_{xx})}{k},$$
where $\lambda_1(u_{xx})\le\cdots\le\lambda_d(u_{xx})$ are the usual eigenvalues of $u_{xx}$ ordered from small to large.
\end{example}

It is also worth mentioning that 
$$\lambda_{\min}[H_m](tI+\gamma)=t+\lambda_{\min}[H_m](\gamma), \quad\forall \gamma\in \mathbb S^d,$$
so we have on $\Gamma'_{H_m}$
$$\operatorname{tr}\big\{(\lambda_{\min}[H_m])_{\gamma^{ij}}(\gamma)\big\}=\frac{d}{dt}\lambda_{\min}[H_m](tI+\gamma)\Big|_{t=0}=1.$$
Consequently (\ref{rhess}) is equivalent to the Dirichlet problem for the  Bellman equation
\begin{equation}
\left\{\begin{array}{rcll}
\displaystyle\inf_{\alpha\in\Gamma'_{H_m}}\big\{(\lambda_{\min}[H_m])_{\gamma^{ij}}(\alpha)u_{x_ix_j}-f\big\}&=&0&\text{in }D\\
u&=&\varphi  &\text{on }\partial D.
\end{array}
\right.
\end{equation}

Then the (bounded and Borel) probabilistic solution $u$ to (\ref{mineigen}) can be defined by (\ref{grp}) and (\ref{grp2}), with
$$a^{ij}(\alpha)=(\lambda_{\min}[H_m])_{\gamma^{ij}}(\alpha),\quad h\equiv1.$$
Applying Theorem \ref{rhess}, we have:
\begin{enumerate}
\item If $f,\varphi\in C^{0,1}(\bar D)$, then $u\in C^{0,1}_{loc}(D)\cap C(\bar D)$, and the first (generalized) derivatives satisfy the estimate (\ref{cfirstma}).
\item If $f,\varphi\in C^{1,1}(\bar D)$, then $u\in C^{1,1}_{loc}(D)\cap C^{0,1}(\bar D)$ and the second (generalized) derivatives satisfy the estimate (\ref{csecondma}). Meanwhile, $u$ is the unique strong solution to (\ref{mineigen}).
\end{enumerate}

The interested readers may establish similar results for complex cases. Also, since 
$$\lambda_{\max}[H_m](M)=-\lambda_{\min}[H_m](-M),$$
similar results can be obtained for the Dirichlet problem (with $f\le0$):
\begin{equation*}
\left\{\begin{array}{rcll}
\lambda_{\max}\big[H_m\big](u_{xx})&=&f  &\text{in }D\\ 
u_{xx}&\in&-\overline\Gamma_{H_m}&\text{in }D\\
u&=&\varphi  &\text{on }\partial D
\end{array}
\right.
\end{equation*}
and its complex counterpart.

\subsection{General Hessian quotient equations} Given an $m$-th degree homogeneous hyperbolic polynomial $H_m$ over $\mathbb S^d$ with respect to $I$, for $0\le k\le m$, we denote by $H^{(k)}_m$ the $k$-th order derivative of $H_m$ in the direction of $I$, i.e.
$$H_m^{(k)}(M)=\frac{d^k}{dt^k}H_m(M+tI)\Big|_{t=0}.$$
We also abbreviate $H_m^{(1)}$ by $H'_m$. We have seen in Section \ref{section2} that $H^{(k)}_m$ are $(m-k)$-th degree homogeneous hyperbolic polynomials, and
$$\Gamma_{H_m}\subset \Gamma_{H'_m}\subset \cdots\subset \Gamma_{H^{(k)}_m}\subset\cdots$$
Here we make use to $H_m^{(k)}$ to construct $\overline{\mathbb S_+^d}$-monotonic and concave functions for which Theorems \ref{rhess} is applicable.

For the sake of completeness, we start from showing the (somewhat well-known) fact that the quotient $H_m/H'_m$ is 
concave on $\Gamma_{H_m}$. 
When $H_m(M)=\sigma_k(M)$, the elementary symmetric function of the eigenvalues of $M$, the concavity of $H_m/H'_m$ on $\Gamma_{H_m}$ has been established in Theorem 6.4, \cite{MR1284912}. The general result can be derived by applying Theorem 8.4, \cite{MR3055586}. To be precise, by using Logarithmic differentiation on $p(t)=H_m(M+tI)$ at $t=0$, we have
\begin{equation}\label{logdiff}
H'_m(\gamma)=H_m(\gamma)\sum_{i=1}^m\frac{1}{\lambda_i(\gamma)}.
\end{equation}
By Theorem 8.4, \cite{MR3055586}, to show that $H_m/H'_m$ is concave on $\Gamma_{H_m}$, it suffices to prove that
$q(x)=1/(\sum_{i=1}^{m}1/x_i)$ is concave on $\mathbb R_+^m$. Notice that $q(x)$ is $1$-homogeneous, so it suffices to show that $q(x)\le q_{(x)}(y)
$
 for all $x,y\in \mathbb R_+^m$, which is indeed the Cauchy-Schwartz inequality
$$\Big(\sum_{i=1}^m1/y_i\Big)^2\le\Big(\sum_{i=1}^{m}1/x_i\Big)\Big(\sum_{i=1}^mx_i/y_i^2\Big).$$
We next denote $H_m/H_m^{(k)}$ by $Q_{m,k}$ and notice that on $\Gamma_{H_m}$
\allowdisplaybreaks
\begin{align*}
Q_{m,k}^{1/k}(\alpha+\beta)=&\Bigg(\prod_{i=1}^k\frac{H_m^{(i-1)}}{H_m^{(i)}}(\alpha+\beta)\Bigg)^{1/k}\\
\ge&\Bigg(\prod_{i=1}^k\Big(\frac{H_m^{(i-1)}}{H_m^{(i)}}(\alpha)+\frac{H_m^{(i-1)}}{H_m^{(i)}}(\beta)\Big)\Bigg)^{1/k}\\
\ge&\Bigg(\prod_{i=1}^k\frac{H_m^{(i-1)}}{H_m^{(i)}}(\alpha)\Bigg)^{1/k}+\Bigg(\prod_{i=1}^k\frac{H_m^{(i-1)}}{H_m^{(i)}}(\alpha)\Bigg)^{1/k}\\
=&Q_{m,k}^{1/k}(\alpha)+Q_{m,k}^{1/k}(\beta),
\end{align*}
which implies that $Q_{m,k}^{1/k}$ is superlinear.

Therefore we can apply Theorem \ref{rhess} with $F=Q_{m,k}^{1/k}$ and $\Gamma=\Gamma_{H_m}$ to study the Dirichlet problem
\begin{equation}\label{quot}
\left\{\begin{array}{rcll}
Q_{m,k}(u_{xx})&=&f^k  &\text{in }D\\ 
u_{xx}&\in&\overline\Gamma_{H_m}&\text{in }D\\
u&=&\varphi  &\text{on }\partial D,
\end{array}
\right.
\end{equation}
provided that $H_m$ is symmetric, $\Gamma_{H_m}$ is elliptic and $D$ is strictly $\Gamma_{H_m}$-convex.
\begin{example}
In particular, if $H_m=\sigma_{m,d}[\det;I]$, then 
$$H_m^{(k)}=k!{d\choose m}\sigma_{m-k,d}[\det;I].$$ 
In this situation,
$$k!{d\choose m}Q_{m,k}(u_{xx})=\frac{\sigma_m(u_{xx})}{\sigma_{m-k}(u_{xx})},\quad \forall 1\le k\le m\le d,$$
where $\sigma_m(u_{xx})$ and $\sigma_{m-k}(u_{xx})$ are the $m$-th and $(m-k)$-th order elementary symmetric functions of the eigenvalues of the symmetric matrix $u_{xx}$ with $\sigma_0\equiv1$. This is the so-called Hessian quotient operator in previous literature.
\end{example}

The (bounded and Borel) probabilistic solution $u$ to (\ref{quot}) can be defined by (\ref{grp}) and (\ref{grp2}), with
\begin{align*}
a^{ij}(\alpha)=&\frac{(H_m)_{\gamma^{ij}}(\alpha)H_m^{(k)}(\alpha)-(H_m^{(k)})_{\gamma^{ij}}(\alpha)H_m(\alpha)}{(H_m)_{(I)}(\alpha)H_m^{(k)}(\alpha)-(H_m^{(k)})_{(I)}(\alpha)H_m(\alpha)},\\
h(\alpha)=&\frac{kH_m^{1-1/k}(\alpha)(H_m^{(k)})^{1+1/k}(\alpha)}{(H_m)_{(I)}(\alpha)H_m^{(k)}(\alpha)-(H_m^{(k)})_{(I)}(\alpha)H_m(\alpha)}.
\end{align*}
Applying Theorem \ref{rhess}, we have:
\begin{enumerate}
\item If $f,\varphi\in C^{0,1}(\bar D)$, then $u\in C^{0,1}_{loc}(D)\cap C(\bar D)$, and the first (generalized) derivatives satisfy the estimate (\ref{cfirstma}).
\item If $f,\varphi\in C^{1,1}(\bar D)$, then $u\in C^{1,1}_{loc}(D)\cap C^{0,1}(\bar D)$ and the second (generalized) derivatives satisfy the estimate (\ref{csecondma}). Meanwhile, $u$ is the unique strong solution to (\ref{quot}).
\end{enumerate}

The interested readers may establish similar results for complex cases.

\subsection{Geometric mean type equations} Another rich set of examples comes from the geometric mean of superlinear functions. Suppose for each $1\le i\le n$, $F^{(i)}$ is superlinear and positive on $\Gamma_{F^{(i)}}$. Define $F=(\prod_{i=1}^nF^{(i)})^{1/n}$. Then on $\Gamma=\cap_{i=1}^n\Gamma_{F^{(i)}}$, we have
\begin{align*}
F(\alpha+\beta)\ge&\bigg(\prod_{i=1}^n\Big(F^{(i)}(\alpha)+F^{(i)}(\beta)\Big)\bigg)^{1/n}\\
\ge&\bigg(\prod_{i=1}^nF^{(i)}(\alpha)\bigg)^{1/n}+\bigg(\prod_{i=1}^nF^{(i)}(\beta)\bigg)^{1/n}=F(\alpha)+F(\beta),
\end{align*}
i.e. the geometric mean of superlinear (positive) functions is still superlinear. 
Consequently, Theorems \ref{rhess} and \ref{chess} are applicable to the Dirichlet problem
\begin{equation}\label{geomean}
\left\{\begin{array}{rcll}
(\prod_{i=1}^nF^{(i)})(u_{xx})&=&f^n  &\text{in }D\\ 
u_{xx}&\in&\overline{\cap_{i=1}^n\Gamma_{F^{(i)}}}&\text{in }D\\
u&=&\varphi  &\text{on }\partial D,
\end{array}
\right.
\end{equation}
provided that for each $1\le i\le n$, $F^{(i)}$ is symmetric, $\Gamma_{F^{(i)}}$ is elliptic and $D$ is strictly $\Gamma_{F^{(i)}}$-convex.

\begin{example}
For the sake of simplicity, we assume that for $1\le k\le d$, $\sigma_k=\sigma_k[\det;I]$ and $\mu_k=\mu_k[\det;I]$. If we let $F^{(i)}=\mu_{k_i}^{1/{d\choose k_i}}$, the partial differential equation in (\ref{geomean}) is
$$\bigg(\mu_{k_1}^{1/{d\choose k_1}}\cdots\mu_{k_n}^{1/{d\choose k_n}}\bigg)(u_{xx})=f^n,$$
where $1\le k_i\le d, \forall 1\le i\le n$.

We see also from the previous subsection that $\sigma_m/\sigma_{m-1}$ can play the role of $F^{(i)}$. Based on this fact  we can construct many possibly quite interesting examples which can fit into (\ref{geomean}), e.g.
$$\prod_{j=1}^{[d/2]}\frac{\sigma_{2j}}{\sigma_{2j-1}}(u_{xx})=f^{[d/2]},$$
$$\frac{\sigma_d^n}{\sigma_{k_1}\cdots\sigma_{k_n}}(u_{xx})=f^{nd-(k_1+\cdots+k_n)},$$
with $1\le k_i\le d, \forall 1\le i\le n$.
\end{example}

The (bounded and Borel) probabilistic solution $u$ to (\ref{geomean}) can be defined by (\ref{grp}) and (\ref{grp2}), with
\begin{align*}
\displaystyle a^{ij}(\alpha)=\frac{\sum_{i=1}^nF_{\gamma^{ij}}^{(i)}\big/F^{(i)}}{\sum_{i=1}^n\operatorname{tr}F_{\gamma^{ij}}^{(i)}\big/F^{(i)}},\quad h(\alpha)=\Bigg(\prod_{i=1}^nF^{(i)}\sum_{i=1}^n\frac{\operatorname{tr}F_{\gamma^{ij}}^{(i)}}{F^{(i)}}\Bigg)^{-1}.
\end{align*}
Applying Theorem \ref{rhess}, we have:
\begin{enumerate}
\item If $f,\varphi\in C^{0,1}(\bar D)$, then $u\in C^{0,1}_{loc}(D)\cap C(\bar D)$, and the first (generalized) derivatives satisfy the estimate (\ref{cfirstma}).
\item If $f,\varphi\in C^{1,1}(\bar D)$, then $u\in C^{1,1}_{loc}(D)\cap C^{0,1}(\bar D)$ and the second (generalized) derivatives satisfy the estimate (\ref{csecondma}). Meanwhile, $u$ is the unique strong solution to (\ref{geomean}).
\end{enumerate}

Again, the interested readers may establish similar results for complex cases.

\subsection{Translation and Calabi-Yau type equations} Theorems \ref{rhess} and \ref{chess} are also applicable to study new problems obtained by translation. This time we take the complex case as the example, and the interested reader can study the corresponding real case.

Let $w$ be a $C^2$ function and $v=u-w$. The function $u$ satisfies (\ref{chess}) (say, in the sense of viscosity solution or strong solution) if and only if $v$ satisfies
\begin{equation}\label{trans}
\left\{\begin{array}{rcll}
G(v_{z\bar z}+w_{z\bar z})&=&f  &\text{in }D\\ 
v_{z\bar z}+w_{z\bar z}&\in&\overline\Theta&\text{in }D\\
u&=&\varphi  &\text{on }\partial D
\end{array}
\right.
\end{equation}
in the same sense. 
\begin{example}
In particular, for any $g=(g_{i\bar j})\in \mathbb H_d$, let $w(z)=g_{i\bar j}z^i\bar z^j$, then $w_{z\bar z}=g$. If we let $G=\det$ and $\Theta=\mathbb H_+^d$, then by Theorem \ref{chess}, we conclude that the Dirichlet problem
\begin{equation}\label{cy}
\left\{\begin{array}{rcll}
\det(u_{i\bar j}+g_{i\bar j})&=&f  &\text{in }D\\ 
u_{i\bar j}+g_{i\bar j}&\ge&0&\text{in }D\\
u&=&\varphi  &\text{on }\partial D
\end{array}
\right.
\end{equation}
has a unique $C^{0,1}_{loc}(D)\cap C(\bar D)$-viscosity solution when $f,\varphi\in C^{0,1}(\bar D)$, and has a unique $C^{1,1}_{loc}(D)\cap C^{0,1}(\bar D)$-strong solution when $f,\varphi\in C^{1,1}(\bar D)$. The system (\ref{cy}) is of simplest Calabi-Yau type.
\end{example}

\section{Acknowledgements}
The author is indebted to Nicolai V. Krylov for helpful  discussions.

\providecommand{\bysame}{\leavevmode\hbox to3em{\hrulefill}\thinspace}
\providecommand{\MR}{\relax\ifhmode\unskip\space\fi MR }
\providecommand{\MRhref}[2]{%
  \href{http://www.ams.org/mathscinet-getitem?mr=#1}{#2}
}
\providecommand{\href}[2]{#2}

\end{document}